\documentclass[12pt]{article}
\usepackage[colorlinks=true,backref=page]{hyperref} 
\usepackage[labelsep=endash]{caption}
\usepackage{graphicx}
\usepackage{latexsym,amssymb}
\usepackage{amsthm}
\usepackage{indentfirst}
\usepackage{amsmath}
\usepackage{color}
\usepackage{xcolor}
\usepackage{mathrsfs}
\usepackage{pxfonts}

\newtheorem{theoremalph}{Theorem}

\newtheorem{corollary-main}[theoremalph]{Corollary}
\newtheorem{Theorem}{Theorem}[section]
\newtheorem*{Theorem A}{Theorem A}
\newtheorem*{Theorem A'}{Theorem A'}

\newtheorem*{Conj*}{Conjecture}

\newtheorem{Definition}[Theorem]{Definition}
\newtheorem{Proposition}[Theorem]{Proposition}
\newtheorem{Lemma}[Theorem]{Lemma}

\newtheorem*{Notation}{Notation}
\newtheorem*{Remark}{Remark}

\newtheorem{Remark-numbered}[Theorem]{Remark}
\newtheorem{Remarks-numbered}[Theorem]{Remarks}
\newtheorem{Corollary}[Theorem]{Corollary}

\newtheorem*{Claim}{Claim}
\newtheorem{claim}[Theorem]{Claim}
\newtheorem{Claim-numbered}{Claim}
\newtheorem*{Acknowledgements}{Acknowledgements}

 \def\BB{{\mathbb B}} 
\def\DD{{\mathbb D}}
  \def\GG{{\mathbb G}}

 \def\JJ{{\mathbb J}} 
\def\LL{{\mathbb L}}
 \def\NN{{\mathbb N}} 
\def\PP{{\mathbb P}}
 \def\RR{{\mathbb R}}

 \def\ZZ{{\mathbb Z}}

\def\cA{{\cal A}}  \def\cG{{\cal G}}

\def\cF{{\cal F}}   \def\cR{{\cal R}}

\def\dim{\operatorname{dim}}

\makeatletter
\renewcommand{\l@section}{\@dottedtocline{2}{3.8em}{3.2em}}
\renewcommand{\l@subsection}{\@dottedtocline{3}{3.8em}{3.2em}}
\newcommand{\subsectionruninhead}{\@startsection{subsection}{2}{0mm}{-\baselineskip}{-0mm}{\bf\large}}
\newcommand{\subsubsectionruninhead}{\@startsection{subsubsection}{3}{0mm}{-\baselineskip}{-0mm}{\bf\normalsize}}
\makeatother
\begin{document}

\title{On the abundance of SRB measures}

\author{Yongluo Cao   \and Zeya Mi \and Dawei Yang\footnote{Y. Cao was partially supported by NSFC (11771317,11790274), Science and Technology Commission of Shanghai Municipality (18dz22710000). Z. Mi was  supported by The Startup Foundation for Introducing Talent of NUIST(Grant No. 2017r070). D.Yang  was partially supported by NSFC 11671288 and NSFC 11790274.}}

\date{\today}

\maketitle

\begin{abstract}
We prove the abundance of Sinai-Ruelle-Bowen measures for diffeomorphisms away from ones with a homoclinic tangency. This is motivated by conjectures of Palis on the existence of physical (Sinai-Ruelle-Bowen) measures for global dynamics. The main novelty in this paper is that we have to deeply study Gibbs $cu$-states in different levels. Note that we have to use random perturbations to give some upper bound of the level of Gibbs $cu$-states.
\end{abstract}

\tableofcontents

\section{Introduction}

The SRB theory was established by Sinai, Ruelle and Bowen in the last seventies to characterize chaotic properties of hyperbolic dynamics in a statistical way \cite{Sin72, Rue76, Bow75, BoR75}. It is a completely beautiful description such that after them, dynamicists want to use similar philosophy to understand dynamics beyond uniform hyperbolicity. In this work, we study the abundance of SRB measures for a large class of diffeomorphisms. This is related to the Palis program for physical (SRB) measures. 

The program of Palis \cite[Page 493]{Pal05} is to characterize global dynamics. As mentioned by Jean-Christophe Yoccoz \cite{Yoc94}: ``Boardly speaking, the goal of the theory of dynamical systems is, as it should be, to understand {\it most} of the dynamics of {\it most} systems''. In \cite[Section 2]{Pal05}, Palis has conjectured that most dissipative diffeomorphisms have finitely many physical (SRB) measures whose basins cover full Lebesgue measure set in the ambient manifold. See also \cite[Page 500]{ShW00}.

There are several definitions of SRB measures from different aspects of interests. We take the one as in Ruelle \cite[Page 8]{Rue96}.
\begin{Definition}\label{Def:SRB}
For a $C^1$ diffeomorphism $f$, an invariant measure $\mu$ of $f$ is said to satisfy the \emph{Pesin's entropy formula} if either $\mu$ has no positive Lyapunov exponents, or it has positive Lyapunov exponents and the entropy of $\mu$  equals to the integral of the sum of positive Lyapunov exponents of $\mu$; an invariant measure $\mu$ is an \emph{Sinai-Ruelle-Bowen measure} if it satisfies the Pesin's entropy formula and has positive metric entropy.
\end{Definition}

The SRB measures in Definition~\ref{Def:SRB} may not be physical. However, in many cases, for example in the setting of Theorem~\ref{maintheorem}, a physical measure is an SRB measure as in Definition~\ref{Def:SRB}. The two notions are very related, and some relationship was studied by Tsujii \cite{Tsu91}.

SRB measures are usually obtained for systems with some hyperbolicity. Newhouse phenomenon \cite{New70,New74,New79}, which is very related to a homoclinic tangency of a hyperbolic periodic orbit, can prevent global hyperbolicity in some robust way. A diffeomorphism $f$ is said to \emph{have a homoclinic tangency} if $f$ has a hyperbolic periodic orbit, whose stable manifolds and unstable manifolds have some non-transverse intersection. Homoclinic tangencies are usually involved in the conjectures of Palis, see \cite{Pal05,PuS00,CrP15,CSY15} for a partial list of references. Let ${\rm Diff}^r(M)$ be the space of $C^r$ diffeomorphisms of $M$. Our main theorem is the following:

\begin{theoremalph}\label{main}
In ${\rm Diff}^1(M)$, any diffeomorphism can be accumulated by one of the following three classes:
\begin{itemize}
\item [---]diffeomorphisms with a homoclinic  tangency;
\item [---]essentially Mores-Smale diffeomorphisms (there exist finitely many sinks such that the union of the basins of these sinks is an open dense set in $M$);
\item [---]diffeomorphisms with SRB measures.
\end{itemize}
\end{theoremalph}

Note that the measure supported on a sink satisfies the Pesin's entropy formula automatically, one has the following corollary:

\begin{corollary-main}\label{entropy-formula}
In ${\rm Diff}^1(M)$, any diffeomorphism can be accumulated by one of the following two classes:
\begin{itemize}
\item [---]diffeomorphisms with a homoclinic  tangency;
\item [---]diffeomorphisms with measures satisfying the Pesin's entropy formula.
\end{itemize}
\end{corollary-main}

For understanding diffeomorphisms away from ones with a homoclinic tangency,  one has to consider a weak form of hyperbolicity, which is called a ``dominated splitting". Let $\Lambda$ be a compact invariant set of a $C^1$ diffeomorphism $f$. For two $Df$-invariant bundles $E,F\subset TM|_{\Lambda}$, we say that \emph{$E$ dominates $F$} or \emph{$F$ is dominated by $E$} if there are constants $C>0$ and $\lambda\in(0,1)$ such that for any point $x\in\Lambda$, we have $\|Df^n|_{F(x)}\|.\|Df^{-n}|_{E(f^n(x))}\|\le C\lambda^n$. Denote the fact that $E$ dominates $F$ by $E\oplus_{\succ} F$. We  say that a compact invariant set $\Lambda$ admits a dominated splitting if there is a $Df$-invariant splitting $TM|_\Lambda=E\oplus_\succ F$ such that $E$ dominates $F$.

For a compact invariant set $\Lambda$, a $Df$-invariant bundle $F$ is \emph{contracted} (by $Df$) if there are constants $C>0$ and $\lambda\in(0,1)$ such that for any point $x$, we have $\|Df^n|_{F(x)}\|\le C\lambda^n$; a $Df$-invariant bundle $F$ is \emph{expanded} (by $Df$) if it is contracted for $f^{-1}$. We say a compact invariant set $\Lambda$ is \emph{partially hyperbolic} if there is a $Df$-invariant splitting $TM|_\Lambda=E^u\oplus_{\succ} E_1^c\oplus_{\succ} \cdots \oplus_{\succ} E_k^c\oplus_{\succ} E^s$ such that $E^u$ is expanded and $E^s$ is contracted. Among partially hyperbolic dynamics, we are more interested in a special type: one requires that each center bundle is one-dimensional. A diffeomorphism $f$ is \emph{partially hyperbolic} if the chain recurrence set of $f$ can be split into finite compact invariant sets such that each set admits a partially hyperbolic splitting whose center bundles are one-dimensional. It has been proved by Crovisier, Sambarino and Yang \cite{CSY15} that any diffeomorphism can be either accumulated by ones with a homoclinc tangency, or accumulated by partially hyperbolic diffeomorphisms.

We will manage to prove the existence of Sinai-Ruelle-Bowen measures on a partially hyperbolic attracting set with one-dimensional dominated center bundles of a $C^2$ diffeomorphism. Note that a compact invariant set $\Lambda$ is \emph{attracting} if there is a neighborhood $U$ of $\Lambda$ {such that $f(\overline{U})\subset U$ and }$\cap_{n\in\NN}f^n(U)=\Lambda$.

\begin{theoremalph}\label{maintheorem}
Assume that $\Lambda$ is an attracting set of a $C^2$ diffeomorphism $f$. If $\Lambda$ admits a partially hyperbolic splitting  $TM|_{\Lambda}=E^u\oplus_{\succ} E_1^c\oplus_{\succ} \cdots \oplus_{\succ} E_k^c\oplus_{\succ} E^s$, where ${\rm dim}E_i^c=1$, for every $1\le i\le k$, $k\ge 1$, then there exists some ergodic SRB measure supported on $\Lambda$.
\end{theoremalph}
The proof of Theorem~\ref{main} is mainly based on Theorem~\ref{maintheorem}.
The main tool to prove Theorem~\ref{maintheorem} is to study Gibbs $cu$-states. Gibbs $u$-states were defined and studied for partially hyperbolic attractors from Pesin and Sinai \cite{PeS82}. It turns out that Gibbs $u$-states have many good properties \cite{PeS82,BDV05}. In contrast to Gibbs $u$-states, Gibbs $cu$-states are defined in the non-uniform case, thus lose some compact property. Moreover, in Theorem~\ref{maintheorem}, there are many center sub-bundles. We have to study Gibbs $cu$-states in different levels. We remark that we have to use random perturbation to give some upper bound of the level of some Gibbs $cu$-states.

Note that the case $k=1$ of Theorem~\ref{maintheorem} has been proved in \cite{CoY05} by using random perturbation and the entropy formula. Liu and Lu \cite{LiL15} obtained SRB measures in a similar philosophy as in \cite{CoY05}.

\begin{Acknowledgements}

We are grateful to J. Buzzi, S. Crovisier, S. Gan, H. Hu, P. Liu, L. Wen, X. Wen and J. Xie for their suggestions and discussions. J. Buzzi and S. Crovisier helped us to check and improve the proof carefully.

\end{Acknowledgements}

\section{Typical dynamics in the $C^1$ topology}

In this section, we will manage to prove Theorem~\ref{main} by using Theorem~\ref{maintheorem}.  Usually one can obtain  SRB measures on some sets with attracting properties. Chain transitivity is a weak form of recurrence. A compact invariant set $\Lambda$ of $f$ is \emph{chain transitive}, if for any $\varepsilon>0$, for any $x,y\in\Lambda$, there are points $x=x_0,x_1,\cdots,x_n=y$ such that $d(f(x_i),x_{i+1})<\varepsilon$ for any $0\le i\le n-1$.
A chain-transitive set $\Lambda$ is a \emph{quasi attractor} if there is a decreasing sequence of attracting set $\{\Lambda_n\}$ such that $\Lambda=\lim_{n\to\infty}\Lambda_n$. For generic diffeomorphisms, we have the following result for quasi attractors, see \cite[Proposition 1.7]{BoC04} and \cite{MoP02}.

\begin{Lemma}\label{Lem:quasi-attractor}
There is a dense $G_\delta$ set $\cR\subset{\rm Diff}^1(M)$ such that for any $f\in\cR$, there is a residual set $R\subset M$ such that for any $x\in R$, the omega-limit set of $x$ w.r.t. $f$ is a quasi attractor.

\end{Lemma}

Crovisier, Sambarino and Yang \cite{CSY15} has proved that for generic diffeomorphisms away from ones with a homoclinic tangency, any chain recurrent class admits a partially hyperbolic splitting whose center bundle can be split into one-dimensional dominated sub-bundles. For quasi attractors, they have more precise information:

\begin{Theorem}\label{Thm:CSY-quasiattractor}
There is a dense $G_\delta$ set $\cR\subset {\rm Diff}^1(M)$ such that for any $f\in\cR$, if $f$ is away from ones with a homoclinic tangency, then for any quasi attractor $\Lambda$ of $f$, when $\Lambda$ is not reduced to be a single periodic orbit, we have that $\Lambda$ admits a partially hyperbolic splitting $TM|_{\Lambda}=E^u\oplus_{\succ} E_1^c\oplus_{\succ} \cdots \oplus_{\succ} E_k^c\oplus_{\succ} E^s$, where $E^u$ is non-trivial and ${\rm dim}E_i^c=1$, for every $1\le i\le k$.

\end{Theorem}

In fact, the main theorem of Crovisier, Pujals and Sambarino \cite{CPS17} gives some information of one-dimensional bundle in a dominated splitting.

\begin{Theorem}\label{Thm:CPS-extremalbundle}[Crovisier-Pujals-Sambarino]
There is a dense $G_\delta$ set $\cR\subset {\rm Diff}^1(M)$ such that for any $f\in\cR$, if a chain transitive set $\Lambda$ of $f$ admits a dominated splitting $TM|_{\Lambda}=E\oplus_\succ F$ satisfying $\dim E=1$, and if $\Lambda$ is not reduced to be a singular periodic orbit, then $E$ is uniformly expanded. Moreover, if $f$ cannot be accumulated by ones with a homoclinic tangency, then $f$ has only finitely many sinks and sources.
\end{Theorem}

Theorem~\ref{Thm:CSY-quasiattractor} can be deduced from Theorem~\ref{Thm:CPS-extremalbundle} and \cite[Theorem 1.1]{CSY15}. This is because by \cite[Theorem 1.1]{CSY15}, there is a dense $G_\delta$ set $\cR\subset {\rm Diff}^1(M)$ such that for any $f\in\cR$, if $f$ is away from ones with a homoclinic tangency, any chain transitive set $\Lambda$ admits a partially hyperbolic splitting $TM|_\Lambda=E^u\oplus_{\succ} E_1^c\oplus_{\succ} \cdots \oplus_{\succ} E_k^c\oplus_{\succ} E^s$ with $\dim E_i^c=1$ for $1\le i\le k$; then by Theorem~\ref{Thm:CPS-extremalbundle}, when $\Lambda$ is not reduced to be a single periodic orbit, we have that $E^u$ is not trivial.

One can also present a proof of Theorem~\ref{Thm:CSY-quasiattractor} from the techniques in \cite{CSY15}.
\begin{proof}[Sketch of the proof of Theorem~\ref{Thm:CSY-quasiattractor}.]
Under the assumptions of Theorem~\ref{Thm:CSY-quasiattractor}, from \cite[Corollary 1.6]{CSY15}, one knows that the quasi attractor $\Lambda$ is a homoclinic class $H(p)$. By \cite[Theorem 1.1]{CSY15}, $\Lambda=H(p)$ admits a partially hyperbolic splitting 
$$TM|_\Lambda=E^u\oplus_\succ E_1^c\oplus_\succ\cdots\oplus_\succ E_k^c\oplus_\succ E^s,~~~\dim E_i^c=1,~\forall 1\le i\le k,$$ 
and the minimal unstable dimension of periodic orbits in $H(p)$ is $\dim E^u$ or $\dim E^u+1$.

Now we argue by contradiction, and assume that $E^u=\{0\}$. Thus, the minimal unstable dimension of periodic orbits in $H(p)$ is  $0$ or $1$. Since $\Lambda=H(p)$ is not reduced to be a single periodic orbit, one knows that the minimal unstable dimension is $1$; moreover, there are periodic orbits in $H(p)$ such that they are weak along $E_1^c$, i.e., their Lyapunov exponents along $E_1^c$ are arbitrarily close to $0$. Thus under some generic assumptions, there is a period point $q$ in $H(p)$ such that the unstable dimension of $q$ is $1$ and its unstable manifold intersect the basin of a sink.  Since the sink cannot be contained in $\Lambda$, one has that the unstable manifold of $p$ cannot be completely contained in $\Lambda$. This gives a contradiction to the fact that $\Lambda$ is a quasi attractor because the unstable set of any point in a quasi attractor is always contained in the quasi attractor.\end{proof}

\smallskip

Now we are ready to prove Theorem~\ref{main}.

\begin{proof}[Proof of Theorem~\ref{main}]

Take a dense $G_\delta$ set $\cR\subset{\rm Diff}^1(M)$ having the properties as in Lemma~\ref{Lem:quasi-attractor}, Theorem~\ref{Thm:CPS-extremalbundle}, Theorem~\ref{Thm:CSY-quasiattractor}.

\smallskip

Since $\cR$ is dense in ${\rm Diff}^1(M)$, it suffices to prove that any $f\in\cR$ has the properties stated in the theorem. To conclude, one can assume that $f$ cannot be accumulated by ones with a homoclinic tangency, and $f$ is not essentially Morse-Smale. We will prove that in this case, $f$ can be accumulated by ones with an SRB measure.

By Lemma~\ref{Lem:quasi-attractor}, there is a dense $G_\delta$ set $R\subset M$ such that for any point $x\in R$, $\omega(x)$ is a quasi attractor. We have two cases:
\begin{itemize}

\item either, for any point $x\in R$, $\omega(x)$ is a trivial quasi-attractor, i.e., it is reduced to be a periodic orbit.

\item or, there is a point $x\in R$ such that $\omega(x)$ is not a trivial quasi attractor.

\end{itemize}

Now we consider the first case. Note that $\omega(x)$ is a periodic sink. By Theorem~\ref{Thm:CPS-extremalbundle}, $f$ has only finitely many sinks. We have that $\cup_{x\in R}\omega(x)$ contains finite sinks and $f$ is essentially Morse-Smale. We get a contradiction.

In the second case, $f$ has a non-trivial quasi attractor. By Theorem~\ref{Thm:CSY-quasiattractor}, the quasi attractor admits a partially hyperbolic splitting $E^u\oplus_\succ E_1^c\oplus_\succ\cdots\oplus_\succ E^c_k\oplus_\succ E^s$ with $\dim E_i^c=1$, where $E^u$ is non-trivial.
By the continuity of the dominated splitting, there is a $C^2$ diffeomropbhism $g$ arbitrarily close to $f$ and an attracting set $\Lambda$ of $g$ such that $TM|_\Lambda=E^u\oplus_\succ E_1^c\oplus_\succ\cdots\oplus_\succ E^c_k\oplus_\succ E^s$ with $\dim E_i^c=1$, where $E^u$ is non-trivial. By Theorem~\ref{maintheorem}, $g$ admits an SRB measure on $\Lambda$.

\end{proof}

\section{Gibbs $u$-states and Gibbs $cu$-states}\label{Sec:gibbs-cu}

In the setting of partial hyperbolicity, a powerful tool to study SRB measures is the \emph{Gibbs $u$-states} which were defined by Pesin-Sinai \cite{PeS82}. For a compact invariant set $\Lambda$ with a partially hyperbolic splitting $TM|_\Lambda=E^{uu}\oplus_\succ E^{cs}$, an invariant measure $\mu$, supported on $\Lambda$ is said to be a \emph{Gibbs $u$-state} (associated to this splitting) if the disintegration along the unstable foliation is absolutely continuous with respect to the Lebesgue measures of these sub-manifolds.

We give a list of properties of Gibbs $u$-states.

\begin{Proposition}\label{Pro:property-u-state}
Assume that $f$ is a $C^2$ diffeomorphism and $\Lambda$ is a compact invariant set of $f$ with a partially hyperbolic splitting $TM|_\Lambda=E^{uu}\oplus_\succ E^{cs}$. Then one has the following properties.
\begin{itemize}

\item The ergodic components of any Gibbs $u$-state are Gibbs $u$-states.

\item The set of Gibbs $u$-states is compact.
\end{itemize}

\end{Proposition}

\begin{proof}
One can see \cite[Lemma 11.13 and Remark 11.15]{BDV05} for instance.
\end{proof}

In this paper, we also have to study a conception called \emph{Gibbs $cu$-states}. Since there are several sub-bundles in this paper, we will use the terminology \emph{Gibbs $E$-state}, for some invariant sub-bundle $E$. 

\begin{Definition}\label{Def:plaque-family}
Assume that $\Lambda$ is a compact invariant set of $f$ and $E\subset TM|_\Lambda$ is an invariant sub-bundle. A \emph{plaque family} of $E$, which is denoted by $\{W^E(x)\}_{x\in\Lambda}$, is a family of embedded sub-manifolds of dimension $\dim E$ satisfying that each sub-manifold is diffeomorphic to the unit ball in $\RR^{\dim E}$, and has the following properties:
\begin{itemize}

\item For any point $x\in\Lambda$, one has $T W^E(x)|_{x}=E(x)$;

\item For any neighborhood $U\subset W^E(f(x))$ of $f(x)$, there is a neighborhood $V$ of $x$ in $W^E(x)$ such that $f(V)\subset U$.

\end{itemize}
Denote by $W^E_\varepsilon(x)$ the $\varepsilon$-neighborhood of $x$ in $W^E(x)$. The second property can be represented as: for any $\varepsilon>0$, there is $\delta>0$ such that for any $x\in\Lambda$, one has $f(W^E_\delta(x))\subset W^E_\varepsilon(f(x))$.

\end{Definition}

For dominated splittings, one has the following plaque family theorem \cite[Theorem 5.5]{HPS77}: 
\begin{Theorem}\label{Thm:plaque-family}
Assume that $\Lambda$ is a compact invariant set with a dominated splitting $TM|_\Lambda =E\oplus_\succ F$. Then there are plaque families of $E$ and $F$. 
\end{Theorem}

One has the existence of unstable manifolds in the dominated case. 

\begin{Lemma}\label{Lem:dominated-unstable}
Assume that $\Lambda$ is a compact invariant set with a dominated splitting $TM|_{\Lambda}=E\oplus_\succ F$.
Given $\ell\in\NN$ and $\lambda\in(0,1)$, there is $\delta=\delta(\ell,\lambda)>0$ such that for any point $x\in\Lambda$, if 
$$\prod_{i=0}^{n-1}\|Df^{-\ell}|_{E(f^{-i\ell}(x))}\|\le\lambda^n,~~~\forall n\in\NN,$$
then $W^E_\delta(x)$ is contained in the unstable manifold of $x$.

\smallskip

Assume that $\mu$ is an ergodic measure supported on $\Lambda$. Assume that all Lyapunov exponents of $\mu$ along $E$ are positive. Then there is a positive $\mu$-measurable function $\delta(x)$ for $\mu$-almost every point $x$ such that $W^E_{\delta(x)}(x)$ is contained in the unstable manifold of $x$.
\end{Lemma}

Lemma~\ref{Lem:dominated-unstable} is a special case of Lemma~\ref{Lem:dominated-unstable-extended} in Section~\ref{Sec:disintegration}.

\smallskip

Using Lemma~\ref{Lem:dominated-unstable}, one can define a measurable partition $\mu$-subordinate to $W^{E,u}$, where $W^{E,u}$ is the unstable manifold tangent to $E$, i.e., $W^{E,u}(x)=W^E(x)\cap W^u_{loc}(x)$.
\begin{Definition}\label{Def:subordinate}

Assume that $\Lambda$ is a compact invariant set with a dominated splitting $TM|_\Lambda=E\oplus_\succ F$. Assume that $\mu$ is an invariant measure satisfying the Lyapunov exponents along $E$ of $\mu$-almost every point $x$ are positive. A measurable partition $\xi$ is said to be \emph{$\mu$-subordinate to $W^{E,u}$} if for $\mu$-almost every point $x$, $\xi(x)$ is an open set contained in $W^{E}_{\delta(x)}(x)$, where $\delta$ is the measurable function as in Lemma~\ref{Lem:dominated-unstable}.

\end{Definition}

\begin{Definition}\label{Def:Gibbscu}
Assume that $f\in {\rm Diff}^2(M)$ has an attractor $\Lambda$ with dominated splitting $TM|_{\Lambda}=E\oplus_{\succ} F$. We say an $f$-invariant (not necessarily ergodic) measure $\mu$ supported on $\Lambda$ is a Gibbs $E$-state if 
\begin{enumerate}
\item
For $\mu$-almost every point, its Lyapunov exponents along $E$ are all positive.

\item the conditional measures of $\mu$  are absolutely continuous  with respect to Lebesgue measures for any measurable partition that is $\mu$-subordinate to $W^{E,u}$.
\end{enumerate}
\end{Definition}

\begin{Proposition}\label{cuproperty} Let $f\in {\rm Diff}^2(M)$ and $\Lambda$ is an attracting set with a dominated splitting $TM|_{\Lambda}=E\oplus_{\succ} F$.
If $\mu$ is a Gibbs $E$-state supported on $\Lambda$, then almost every ergodic component of $\mu$ is a Gibbs $E$-state.
\end{Proposition}
\begin{proof}
Since the Lyapunov exponents of $\mu$-almost every point along $E$ are all positive, one has that the Lyapunov exponents along $E$ of any ergodic component $\nu$ of $\mu$ are all positive. 

Consider an ergodic component $\nu$ of $\mu$. From \cite[Chapter IV, Remark 2.1]{LiQ95}, it suffices to prove that there is one measurable partition  $\nu$-subordinate to $W^{E,u}$ such that the conditional measures of $\nu$ are absolutely continuous with respect to Lebesgue measures. Any measurable partition $\mu$-subordinate to $W^{E,u}$ gives such kind of measurable partitions of $\nu$. Moreover, by the Birkhoff ergodic theorem, there is a set $R$ with full $\mu$-measure such that the intersection of $R$ with almost every unstable manifold $W^{E,u}$ is the set of typical points for one ergodic component of $\mu$. Thus, the conditional measures of $\nu$ are absolutely continuous with respect to Lebesgue measures. See also \cite[Section 6]{LeY85}.
\end{proof}

\begin{Notation}
Let $\Lambda$ be a compact invariant set with a partially hyperbolic splitting $TM|_\Lambda=E^u\oplus_\succ E_1^c\oplus_\succ \cdots \oplus_\succ E_k^c\oplus_\succ E^s$, $\dim E_i^c=1$ for $1\le i\le k$. For any ergodic measure $\mu$ supported on $\Lambda$, denote by $\lambda_i^c(\mu)$ the Lyapunov exponent of $\mu$ along $E_i^c$ for $1\le i\le k$.

\end{Notation}

For the splitting in Theorem~\ref{maintheorem}, one can define some index for Gibbs $cu$-states.
Assume that $\Lambda$ is an attracting set of a $C^2$ diffeomorphism $f$ with a partially hyperbolic splitting $TM|_\Lambda=E^u\oplus_\succ E_1^c\oplus_\succ \cdots \oplus_\succ E_k^c\oplus_\succ E^s$, $\dim E_i^c=1$ for $1\le i\le k$. Given $0\le i\le k$, denote by $\cG_i$ the set of Gibbs $E^u\oplus_\succ E_1^c\oplus_\succ\cdots\oplus_\succ E_i^c$-states. By convention, $\cG_0$ is the set of Gibbs $u$-states. 

As a direct consequence of Proposition~\ref{cuproperty}, one has the following corollary, whose proof is omitted.
\begin{Corollary}\label{Cor:ergodic-in-i}
Given $0\le i\le k$, if $\mu\in\cG_i$, then $\nu\in\cG_i$ for any ergodic component $\nu$ of $\mu$.

\end{Corollary}

By using some absolute continuity of unstable sub-foliation, one has the following result, whose proof is contained in Appendix~\ref{App:invariant-manifold}.

\begin{Proposition}\label{Proposition:several-G}
 We have that $\cG_0\supset \cG_1\supset\cdots\supset \cG_k$.
\end{Proposition}

The limit measure of a sequence of measures in $\cG_i$ may not be contained in $\cG_i$ if $i>0$. However, one has the following criterion, whose proof is given in Section~\ref{Sec:prove-the-left}.
\begin{Theorem}\label{Thm:criterion-in-i}
Assume that $\Lambda$ is an attracting set of a $C^2$ diffeomorphism $f$ with a partially hyperbolic splitting $TM|_\Lambda=E^u\oplus_\succ E_1^c\oplus_\succ \cdots\oplus_\succ E_k^c\oplus_\succ E^s$ with $\dim E_i^c=1$, $1\le i\le k$. Assume that $\{\mu_n\}\subset \cG_i$ is a sequence of ergodic measures and $\lim_{n\to\infty}\mu_n=\mu$. If there is $\alpha>0$ such that for $\lambda_i^c(\nu)\ge\alpha>0$ for any ergodic component $\nu$ of $\mu$, then $\mu\in\cG_i$.

\end{Theorem}

\begin{Definition}\label{Def:disintegration-index}

For the measure $\mu\in\cG_0$, denote by $I(\mu)$ the maximal $i$ such that $\mu\in\cG_i$. One can call this $I(\mu)$ is \emph{disintegration index} of $\mu$, although we will not mention it again.

\end{Definition}

We have the following simple observation:
\begin{Lemma}\label{Lem:component-srb}
Assume that $\Lambda$ is an attracting set of a $C^2$ diffeomorphism $f$ and $\Lambda$ admits a partially hyperbolic splitting $TM|_\Lambda=E^u\oplus_\succ E_1^c\oplus_\succ\cdots\oplus_\succ E_k^c\oplus_\succ E^c$ with $\dim E_i^c=1$ for $1\le i\le k$.
For an invariant measure $\mu$, assume that $I(\mu)=i$. Then we have
\begin{enumerate}
\item if $\mu$ has an ergodic component $\nu$ satisfying $\lambda_{i+1}^c(\nu)\le 0$, then $\nu$ is an SRB measure.

\item If $\int \log\|Df|_{E^c_{i+1}}\|{\rm d}\mu\le 0$, then the ergodic components of $\mu$ contains an SRB measure

\end{enumerate}

\end{Lemma}

\begin{proof}
By Corollary~\ref{Cor:ergodic-in-i},  if $\nu$ is one ergodic component of $\mu$, then we have that $I(\nu)\ge I(\mu)$. Hence by Proposition~\ref{Proposition:several-G}, $\nu\in\cG_i$. Thus, if $\lambda_{i+1}^c(\nu)\le 0$, then $\nu$ is an SRB measure by the classical result \cite{LeY85}. Thus the first item is proved.

For the second item, one notices that if $\int \log\|Df|_{E^c_{i+1}}\|{\rm d}\mu\le 0$, then there is an ergodic component $\nu$ of $\mu$ satisfying $\lambda_{i+1}^c(\nu)\le 0$. Thus $\nu$ is an SRB measure by the first item.
\end{proof}

One considers a special subset $\cG_i^0\subset \cG_i$ such that $\mu\in\cG_i^0$ if and only if $\mu\in\cG_i$, $\lambda_{i+1}^c(\nu)>0$ for any ergodic component $\nu$ of $\mu$, and there is a sequence of measures $\nu_n$ in the ergodic components of $\mu$ such that $\lim_{n\to\infty}\lambda_{i+1}^c(\nu_n)=0$. Note that $\cG_i^0$ may be an empty set for any $0\le i\le k$.

\begin{Theorem}\label{Thm:non-empty}
Assume that $\Lambda$ is an attracting set of a $C^2$ diffeomorphism $f$ and $\Lambda$ admits a partially hyperbolic splitting $TM|_\Lambda=E^u\oplus_\succ E_1^c\oplus_\succ\cdots\oplus_\succ E_k^c\oplus_\succ E^c$ with $\dim E_i^c=1$ for $1\le i\le k$. Then we have that either $f$ has an SRB measure supported on $\Lambda$, or there is $0\le i\le k$ such that $\cG_i^0\neq\emptyset$.

\end{Theorem}
The proof of Theorem~\ref{Thm:non-empty} will use random perturbations, we will give its proof by Theorem~\ref{Thm:random-limit-G} and give the proof of Theorem~\ref{Thm:random-limit-G} in Section~\ref{Sec:prove-the-left}.

\begin{Theorem}\label{Thm:minimal-obtain-srb}\footnote{S. Crovisier helped us to clean some ideas of Theorem~\ref{Thm:minimal-obtain-srb}.}
Assume that $\Lambda$ is an attracting set of a $C^2$ diffeomorphism $f$ and $\Lambda$ admits a partially hyperbolic splitting $TM|_\Lambda=E^u\oplus_\succ E_1^c\oplus_\succ\cdots\oplus_\succ E_k^c\oplus_\succ E^c$ with $\dim E_i^c=1$ for $1\le i\le k$.
Choose $0\le i\le k$ satisfying $\cG_i^0\neq\emptyset$ and $\cG_j^0=\emptyset$ for any $j<i$. For any $\mu\in\cG_i^0$, taking $\{\nu_n\}$ a sequence of ergodic components of $\mu$ satisfying $\lim_{n\to\infty}\lambda_{i+1}^c(\nu_n)=0$. Then there is an ergodic component $\eta$ of $\nu=\lim_{n\to\infty}\nu_n$ such that $\eta$ is an SRB measure.

\end{Theorem}

\begin{proof}
By the properties of Gibbs $u$-states (Proposition~\ref{Pro:property-u-state}), we know that any $\nu_n$ and $\nu=\lim_{n\to\infty}\nu_n$ are Gibbs $u$-states, i.e., $\nu\in\cG_0$. Thus $I(\nu)$ can be defined. Since $\lim_{n\to\infty}\lambda_{i+1}^c(\nu_n)=0$, we have that
$$\int \log\|Df|_{E^{c}_{i+1}}\|{\rm d}\nu=0.$$
This implies that $I(\nu)\le i$. 

\begin{claim}

We have that either $I(\nu)=i$, or one ergodic component of $\nu$ is an SRB measure.

\end{claim}
\begin{proof}[Proof of the Claim]
Assume that the conclusion of this claim is not true, i.e. $I(\nu)=j<i$ and there is no SRB measures in the ergodic components of $\nu$. Thus, by Lemma~\ref{Lem:component-srb}, we have that $\lambda_{j+1}^c(\eta)>0$ for any ergodic component $\eta$ of $\nu$.

By the minimality of $i$, we have that there is a constant $\alpha>0$ such that $\lambda_{j+1}^c(\eta)>\alpha>0$ for any ergodic component $\eta$ of $\nu$. Otherwise, we have that $\cG_j^0\neq\emptyset$ and give a contradiction to the minimality of $i$.

By Theorem~\ref{Thm:criterion-in-i}, we have that $\nu\in \cG_{j+1}$. This contradicts to the fact that $I(\nu)=j$.

\end{proof}

Under the condition that $I(\nu)=i$, then by Lemma~\ref{Lem:component-srb}, the ergodic components of $\nu$ contains an SRB measure since we have that $\int \log\|Df|_{E^{c}_{i+1}}\|{\rm d}\nu=0.$ Thus one can conclude by applying the above Claim.

\end{proof}
\begin{proof}[Proof of Theorem~\ref{maintheorem}] Under the setting of Theorem~\ref{maintheorem}, by Theorem~\ref{Thm:non-empty}, either there is an SRB measure supported on $\Lambda$, or there is $i$ such that $\cG_i^0\neq\emptyset$. 

Now we consider the case that $\cG_i^0\neq\emptyset$ for some $i$. Take a minimal $i$ with this property, i.e. $\cG_i^0\neq\emptyset$ but $\cG_j^0=\emptyset$ for any $j<i$. Then by Theorem~\ref{Thm:minimal-obtain-srb}, one can also get an SRB measure. Thus the proof of Theorem~\ref{maintheorem} is complete.
\end{proof}

We will give the proofs of Theorem~\ref{Thm:criterion-in-i} and Theorem~\ref{Thm:non-empty} in next sections. Note that Theorem~\ref{Thm:criterion-in-i} is used to prove Theorem~\ref{Thm:minimal-obtain-srb}.

\section{Random dynamical systems and random perturbations}

The main issue for proving Theorem~\ref{maintheorem} is to do some random perturbation for a deterministic dynamical system. One can see fundamental knowledge of random dynamical systems and random perturbations in \cite{Kif86,Kif88,LiQ95}.

\bigskip
Recall that ${\rm Diff}^r(M)$ is the space of $C^r$ diffeomorphisms.
\begin{Definition}\label{Def:random-dynamical-system}

Let $\Omega$ be a compact metric space, $\ell:\Omega\to{\rm Diff}^2(M)$ be a continuous map. Denote by $f_\omega=\ell(\omega)$
for each $\omega\in\Omega$.

For each $\underline\omega=(\cdots,\omega_{-1},\dot\omega_0,\omega_1,\cdots)\in\Omega^\ZZ$, it defines a sequence of diffeomorphisms $f_{\underline\omega}=\{\cdots,f_{\omega_{-1}},\dot f_{\omega_0},f_{\omega_1},\cdots\}$. A point in $\Omega^\ZZ\times M$ is denoted by $[\underline\omega,x]$.

One can thus define an extended dynamical system on a compact metric space $\Omega^\ZZ\times M$ in the following way:
\begin{align*}
G: \Omega^{\mathbb{Z}}\times M & ~~\longrightarrow ~~\Omega^{\mathbb{Z}}\times M \\
[\underline{\omega}~,~x] &~~\longmapsto ~~[\sigma(\underline{\omega}),f_{\omega_0}(x)],
\end{align*}
where $\sigma$ is the left shift operator on the space $\Omega^\ZZ$.

We say that $G$ is an \emph{extended dynamical system} generated by $(\Omega,\ell)$. When there is a Borel probability $\nu$ on $\Omega$, then $G$ is also called a \emph{random dynamical system} with randomness $\nu$, or $(G,\nu)$ is a random dynamical system generated by $(\Omega,\ell,\nu)$.
\end{Definition}

When $\Omega$ is reduced to be a point, the extended dynamical system $G$ can be identical to be the dynamical system of a diffeomorphism.

We will consider stationary measures of a random dynamical system. 

\begin{Definition}\label{Def:stationary-measure}
For a measure $\nu$ supported on $\Omega$, a measure $\mu$ supported on $M$ is called a \emph{stationary measure}  of $\nu$ if for any Borel set $A$, we have
$$\mu(A)=\int \mu(f_\omega^{-1}(A)){\rm d}\nu(\omega).$$

\end{Definition}

\begin{Remark}
The measure $\mu$ is in fact said to be the stationary measure of a random process generated by $\Omega$, $\ell$ and $\nu$. One can see \cite[Chapter I]{Kif86} for the discussion of the random process.

\end{Remark}

A Borel set $A$ is called \emph{randomly invariant} (for $\nu$ and $\mu$) if for $\mu$-almost every $x$, we have
{
$$
x\in A\quad \textrm{implies}\quad  f_\omega(x)\in A,\quad \nu-a.e.\quad \omega;
$$
$$
x\notin A\quad \textrm{implies}\quad  f_\omega(x)\notin A,\quad \nu-a.e.\quad \omega.
$$
}

A stationary measure $\mu$ is \emph{ergodic} if for any randomly invariant set $A$, we have that $\mu(A)=0$ or $\mu(A)=1$.

\begin{Theorem}\label{Thm:ergodic-stationary-exist}

Ergodic stationary measure for $\nu$ always exists.  

\end{Theorem}

\begin{proof}

The proof follows from the existence of stationary measures (\cite[Lemma 2.2]{Kif86} and \cite[Proposition 5.6]{Via14}) and the ergodic decomposition theorem of stationary measures \cite[Appendex A.1]{Kif86} and \cite[Theorem 5.14]{Via14}).
\end{proof}

The map $\ell:\Omega\to {\rm Diff}^2(M)$ in fact induces a map from $\Omega\times M$ to $M$, which is also denoted by $\ell$:
\begin{align*}
\ell: \Omega\times M & ~~\longrightarrow ~~M \\
(\omega~,~x) &~~\longmapsto ~~f_\omega(x).
\end{align*}
Thus for any $x\in M$, one obtains a map $\ell_x:\Omega\to M$. For any measure $\nu$ supported on $\Omega$, one has the measure $(\ell_x)_*\nu$ on $M$:
$$(\ell_x)_*\nu(A)=\nu(\ell_x^{-1}(A)).$$

A random dynamical system $(G,\nu)$ generated by $(\Omega,\ell,\nu)$ is \emph{regular} if for any $x\in M$, $(\ell_x)_*\nu$ is absolutely continuous with respect to the Lebesgue measure. Regular random dynamical systems have the following good property. The proof is folklore and is omitted here. 
\begin{Lemma}\label{Lem:regular-lebesgue}
If a random dynamical system is regular, then any stationary measure is absolutely continuous with respect to Lebesgue.
\end{Lemma}

\begin{Definition}\label{Def:random-perturbation}
A sequence of random dynamical systems $\{(G,\nu_n)\}_{n\in\NN}$ generated by  $\{(\Omega,\ell,\nu_n)\}_{n\in\NN}$ is \emph{nested} if ${\rm supp}(\nu_{n+1})\subset {\rm supp}(\nu_n)$ for any $n\in\NN$. For a diffeomorphism $f$, a nested sequence of regular random dynamical systems $\{(G,\nu_n)\}_{n\in\NN}$ generated by $(\Omega,\ell,\nu_n)\}$ is a \emph{random perturbation} of $f$ if $\lim_{n\to\infty}{\rm supp}(\nu_n)=\{\omega\}$ such that $\ell(\omega)=f$.
\end{Definition}

\begin{Theorem}\label{Thm:perturbation-existence}
For any $C^2$ diffeomorphism $f$, there is a regular random perturbation of $f$.
\end{Theorem}
The proof of Theorem~\ref{Thm:perturbation-existence} is classical and contained in \cite[Page 1120]{CoY05}. The idea is to find (possibly many) vector fields $X^1,X^2,\cdots,X^k$ on $M$ such that they span the tangent space everywhere. Then we take $\Omega=[-1,1]^d$ and $v_n$  the normalized Lebesgue measure on $[-1/n,1/n]^d$. The composition $\varphi^1_{t_1}\circ\varphi^2_{t_2}\cdots\circ\varphi^k_{t_k}\circ f$ gives a regular random perturbation of $f$, where $\varphi^i$ is the flow generated by $X^i$ for $1\le i\le k$.

\bigskip

The following proposition could be seen as an exercise.
\begin{Proposition}\label{Pro:limit-in-lambda} 
Let $\{(G,\nu_n)\}_{n\in\NN}$ be a random perturbation of a diffeomorphism $f$. If $\mu_{n}$ is a stationary measure of $(G,\nu_n)$, then all accumulation points of $\{\mu_n\}$ are $f$-invariant measures. Moreover, if $\mu_n$ is contained in a small neighborhood of $\Lambda$, then $\mu$ is an invariant measure supported on $\Lambda$.
\end{Proposition}

In this paper, we will consider the limit of a sequence of ergodic stationary measures of a regular perturbation of $f$. The limit measure is not necessarily ergodic. However, we will call it an \emph{ergodic limit}.

\begin{Definition}\label{Def:ergodiclimit}
For an invariant measure $\mu$ of a $C^2$ diffeomorphism $f$, if there is a regular random perturbation $\{(G,\nu_n)\}_{n\in\NN}$ of $f$ such that there is a sequence of ergodic stationary measure $\mu_{n}$ of $(G,\nu_n)$, and 
$$\mu=\lim_{n\to\infty}\mu_{n},$$
then $\mu$ is said to be a \emph{randomly ergodic limit}.

\end{Definition}

One has the following extended version of Theorem~\ref{Thm:non-empty}.

\begin{Theorem}\label{Thm:random-limit-G}
Assume that $\Lambda$ is an attracting set of a $C^2$ diffeomorphism $f$ and $\Lambda$ admits a partially hyperbolic splitting $TM|_{\Lambda}=E^u\oplus_\succ E_1^c\oplus_\succ\cdots\oplus_\succ E_k^c\oplus_\succ E^s$ with $\dim E_j^c=1$, $1\le j\le k$. Assume that $\mu$ is a randomly ergodic limit supported on $\Lambda$, then either there is an ergodic component $\nu$ of $\mu$ such that $\nu$ is an SRB measure, or there is $0\le i\le k$ such that $\mu\in \cG_i^0$.

\end{Theorem}

One can give the proof of Theorem~\ref{Thm:non-empty} by assuming Theorem~\ref{Thm:random-limit-G}.

\begin{proof}[Proof of Theorem~\ref{Thm:non-empty}]

By Theorem~\ref{Thm:perturbation-existence}, there is a sequence of regular random perturbation $\{(G,\nu_n)\}_{n\in\NN}$ of $f$. By Theorem~\ref{Thm:ergodic-stationary-exist}, each $(G_n,\nu_n)$ has an ergodic stationary measure $\mu_n$. After a subsequence, one can assume that $\{\mu_n\}$ converges to a measure $\mu$. By Proposition~\ref{Pro:limit-in-lambda}, $\mu$ is a randomly ergodic limit supported on $\Lambda$. By Theorem~\ref{Thm:random-limit-G}, 
\begin{itemize}
\item either there is an ergodic component $\nu$ of $\mu$ such that $\nu$ is an SRB measure, thus there is an SRB measure supported on $\Lambda$,

\item or $\mu\in\cG_i^0$, in other words, $\cG_i^0\neq\emptyset$ for some $0\le i\le k$.

\end{itemize}

The proof of Theorem~\ref{Thm:non-empty} is complete.
\end{proof}

It remains to prove Theorem~\ref{Thm:criterion-in-i} and Theorem~\ref{Thm:random-limit-G} in next sections.

\section{Good approximations of Pesin blocks}

We define some canonical projections on $\Omega^\ZZ\times M$:
$${\PP_M:~\Omega^\ZZ\times M\to M},~~\PP_+:~\Omega^\ZZ\times M\to\Omega^{\NN\cup\{0\}}\times M.$$

\subsection{The lifted measure of a stationary measure}

\begin{Lemma}\label{Lem:invariant-measure-G}
Let $G$ be the extended dynamical system generated by $(\Omega,\ell)$. For any Borel probability $\nu$ and any its stationary measure $\mu$, there is a unique $G$-invariant Borel probablity measure $\mu^G$ supported on $\Omega^\ZZ\times M$ such that $(\PP_+)_*\mu^G=\nu^{\NN\cup\{0\}}\times \mu$.

Consequenly, we have the following properties:
\begin{itemize}
\item $\mu$ is an ergodic stationary measure of $\nu$ if and only if $\mu^G$ is ergodic for $G$.

\item Assume that $\mu_n$ is the stationary measure of $\nu_n$ for any $n\in\NN\cup\{0\}$ and $\lim_{n\to\infty}\mu_n=\mu_0$, $\lim_{n\to\infty}\nu_n=\nu_0$, then $\lim_{n\to\infty}\mu_n^G=\mu_0^G$.

\end{itemize}
\end{Lemma}

\begin{proof} By \cite[Proposition 1.2 and Proposition 1.3]{LiQ95}, one knows the existence and uniqueness of $\mu^G$, and the fact that $\mu$ is an ergodic stationary measure of $\nu$ if and only if $\mu^G$ is an ergodic measure of $G$. 

Assume that $\lim_{n\to\infty}\nu_n=\nu_0$, $\lim_{n\to\infty}\mu_n=\mu_0$. Assume that $\eta=\lim_{n\to\infty}\mu_n^G$. It suffices to prove that $\eta=\mu_0^G$. Since $\mu_n^G$ is invariant for any $n\in\NN$, one has that $\eta$ is $G$-invariant. By the continuity of the projection $\PP_+$, one has that
$$(\PP_+)_*(\eta)=\lim_{n\to\infty}(\PP_+)_*(\mu_n^G)=\lim_{n\to\infty}\nu_n^{\NN\cup\{0\}}\times\mu_n=\nu_0^{\NN\cup\{0\}}\times\mu_0.$$
Thus, by the uniqueness of $\mu_0^G$, one has that $\eta=\mu_0^G$.\end{proof}

As a consequence of Lemma~\ref{Lem:invariant-measure-G}, one has the following result on lifted measures. The proof is omitted.

\begin{Corollary}\label{Cor:lifted-limit}
Let $G$ be the extended dynamical system generated by $(\Omega,\ell)$. Assume that there is $\omega_f\in\Omega$ such that $\ell(\omega_f)=f$. One has the following property.
\begin{itemize}

\item If $\{(G,\nu_n)\}_{n\in\NN}$ is a random perturbation of $f$, and $\{\mu_n\}_{n\in\NN}$ are the stationary measures of $\{\nu_n\}_{n\in\NN}$, $\lim_{n\to\infty}\mu_n=\mu$, then $\lim_{n\to\infty}\mu_n^G=\mu^G=\delta_{\omega_f}^\ZZ\times\mu$.


\end{itemize}
\end{Corollary}

\subsection{Dominated splittings for random dynamical systems}

We want to present the dynamics of $G$. For any $\underline\omega=(\cdots,\omega_{-1},\dot\omega_0,\omega_1,\cdots)\in\Omega^\ZZ$ and any $x\in M$, one defines
\begin{itemize}

\item $f_{\underline\omega}^n(x)=f_{\omega_{n-1}}\circ \cdots\circ f_{\omega_0}(x)$, if $n\ge 1$,

\item $f^0=id$,

\item $f_{\underline\omega}^n(x)=f^{-1}_{\omega_{n}}\circ \cdots\circ f^{-1}_{\omega_{-1}}(x)$, if $n\le -1$.

\end{itemize}
For the presentation, we have
$$G^n([\underline\omega,x])=[\sigma^n(\underline\omega),f^n_{\underline\omega}(x)],~~~\forall n\in\ZZ.$$

\smallskip

One has to associate a tangent bundle for any compact $G$-invariant set $\Lambda^G$ in $\Omega^\ZZ\times M$ for the extended dynamical system $G$.

\begin{Definition}\label{Def:dominated-extended}
For each $[\underline\omega,x]$, we can attach a vector space $TM|_{[\underline\omega,x]}=TM|_{\PP_M([\underline\omega,x])}=TM|_x$. This gives a vector bundle on $\Omega^\ZZ\times M$. This vector bundle is also called the \emph{tangent bundle}, and is also denoted by $TM$.
 
 \smallskip
 
 A map $DG:~TM|_{\Omega^\ZZ\times M}\to TM|_{\Omega^\ZZ\times M}$ can be defined by $	DG(v)=Df_{\omega_0}(v)\in TM|_{f_{\omega_0}(x)}$ for every $v\in TM_{[\underline{\omega},x]}$.

For a $G$-invariant set $\Lambda^G$ in $\Omega^\ZZ\times M$, a sub-bundle $E\subset TM|_{\Lambda^G}$ is said to be \emph{invariant} or \emph{$DG$-invariant} if $DG(E([\underline\omega,x]))=E(G([\underline\omega,x]))$ for any $[\underline\omega,x]\in\Lambda^G$.

\smallskip

A $DG$-invariant splitting $TM|_{\Lambda^G}=E\oplus_\succ F$ on a compact $G$-invariant set $\Lambda^G$ is a \emph{dominated splitting} if there are constants $C>0$ and $\lambda\in(0,1)$ such that for any $[\underline\omega,x]\in\Lambda^G$ and any $n\in\NN$, we have that
$$\|DG^n|_{F([\underline\omega,x])}\|\|DG^{-n}|_{E(G^n([\underline\omega,x]))}\|\le C\lambda^n.$$
\end{Definition}

The following proposition is standard. One can see its proof in \cite[Corollary 2.8]{CrP18} for instance.
\begin{Proposition}\label{Pro:extend-dominated}
Assume that a compact invariant set $\Lambda^G\subset \Omega^\ZZ\times M$ of $G$ admits a dominated splitting $TM|_{\Lambda^G}=E\oplus_\succ F$. Then there is a neighborhood $U^G$ of $\Lambda^G$ such that the maximal $G$-invariant in $U^G$ also admits a dominated splitting with the same type of $E\oplus_\succ F$.

\end{Proposition}

We can lift bundles of one diffeomorphism to the extended dynamical system. The result if folklore.

\begin{Lemma}\label{Lem:extend-bundle}
Let $G$ be the extended dynamical system generated by $(\Omega,\ell)$. Assume that there is $\omega_f\in\Omega$ such that $\ell(\omega_f)=f$. Then,
\begin{itemize}
\item If $\mu$ is an $f$-invariant measure, then $\mu^G$ has the same Lyapunov exponents of $G$ as $\mu$ and $f$.

\item If $\Lambda$ is a compact invariant set, then $\Lambda^G=\{\omega_f\}^\ZZ\times\Lambda$ is a compact invariant set of $G$. Moreover, if $\Lambda$ admits a dominated splitting $TM|_\Lambda=E\oplus_\succ F$ with respect to $Df$, then $\Lambda^G$ admits a dominated splitting with respect to $DG$ of the same type.

\end{itemize}

\end{Lemma}

\subsection{The Pesin blocks for the extended dynamical systems}

Assume that a compact $G$-invariant set $\Lambda^G\subset\Omega^\ZZ\times M$ and $E\subset TM|_{\Lambda^G}$ is an invariant sub-bundle. We define the following subset of $\Lambda^G$: given a constant $\alpha>0$ and an integer $\ell\in\NN$,
$$\Lambda^G_\ell(E,\alpha)=\{[\underline\omega,x]\in\Lambda^G:\prod_{i=0}^{n-1}\|DG^{-\ell}|_{E(G^{-i \ell}([\underline\omega,x]))}\|\le {\rm e}^{-\alpha \ell n},~\forall n\in\NN\}.$$
One can also consider finite pieces of orbits:
$$\Lambda^G_{\ell,n}(E,\alpha)=\{[\underline\omega,x]\in\Lambda^G:\prod_{i=0}^{m-1}\|DG^{-\ell}|_{E(G^{-i \ell}([\underline\omega,x]))}\|\le {\rm e}^{-\alpha \ell n},~\forall 1\le m\le n\}.$$

It is clear that 
$$\Lambda^G_\ell(E,\alpha)=\bigcap_{n\in\NN}\Lambda^G_{\ell,n}(E,\alpha).$$

When $E$ and $F$ are invariant sub-bundles over $\Lambda^G$ and $F$ is dominated by $E$, we do not distinguish $\Lambda^G_\ell(F,\alpha)$ and $\Lambda^G_\ell(E\oplus F,\alpha)$ although there could be some slight differences on constants. Note that we do not assume that $E\oplus F=TM|_{\Lambda^G}$.

For the extended dynamical systems, one has the following result:

\begin{Proposition}\label{Pro:uniform-pesin-block-random-perturbation}
Assume that $E$ is a one-dimensional continuous $DG$-invariant sub-bundle over a compact $G$-invariant set $\Lambda^G\subset\Omega^\ZZ\times M$.  Assume that $\eta$ supported on $\Lambda^G$ is a $G$-invariant measure, and there are  constants $\theta>\alpha>0$ such that $\int \log\|DG|_{E}\|{\rm d}\zeta>\theta$ for any ergodic component $\zeta$ of $\eta$.

If $\{\eta_n\}$ is a sequence of ergodic measures of $G$ such that $\lim_{n\to\infty}\eta_n=\eta$, then for any $\varepsilon>0$, there is $\ell=\ell(\varepsilon)>0$ such that
$$\liminf_{n\to\infty}\eta_n(\Lambda^G_\ell(E,\alpha))>1-\varepsilon.$$

\end{Proposition}

One has to do some preparations. One can find the constant $\ell\in\NN$ by the following lemma:

\begin{Lemma}\label{Lem:find-integer}
Assume that $E$ is a one-dimensional continuous $DG$-invariant sub-bundle over a compact $G$-invariant set $\Lambda^G\subset\Omega^\ZZ\times M$.  Assume that $\eta$ supported on $\Lambda^G$ is a $G$-invariant measure, and there are  constants $\theta>\alpha>0$ such that $\int \log\|DG|_{E}\|{\rm d}\zeta>\theta$ for any ergodic component $\zeta$ of $\eta$. 
%
Then for any $\delta>0$, there is $\ell=\ell(\delta)\in\NN$ such that
$$\eta(\Lambda^G_{\ell,1}(E,\alpha))>1-\delta.$$
\end{Lemma}
\begin{proof}
Since $\dim E=1$ and $E$ is continuous, one has that for $\eta$-almost every point $[\underline\omega,x]$, 
$$\lim_{n\to\infty}\frac{1}{n}\sum_{i=0}^{n-1}\log\|DG^{-1}|_{E(G^{-i}([\underline\omega,x]))}\|\le -\theta.$$
Thus for any $\delta>0$, there is $\ell=\ell(\delta)$ such that
$$\eta(\{[\underline\omega,x]:~\frac{1}{n}\sum_{i=0}^{n-1}\log\|DG^{-1}|_{E(G^{-i}([\underline\omega,x]))}\|\le -\alpha,~\forall n\ge\ell\})>1-\delta.$$
It is clear that $\{[\underline\omega,x]:~\frac{1}{n}\sum_{i=0}^{n-1}\log\|DG^{-1}|_{E(G^{-i}([\underline\omega,x]))}\|\le -\alpha,~\forall n\ge\ell\}\subset\Lambda^G_{\ell,1}(E,\alpha)$ since $\dim E=1$ . Thus one can conclude.
%
\end{proof}

For the proof of Proposition~\ref{Pro:uniform-pesin-block-random-perturbation}, one needs a recent Pliss lemma in \cite{AnV17}. One can see a proof of Lemma~\ref{Lem:pliss-like} in Appendix~\ref{Sec:pliss-like}. 
\begin{Lemma}\label{Lem:pliss-like}
For any $\gamma_1<\gamma_2\le \max\{0,\gamma_2\}<C$, for any $\varepsilon>0$, there is $\rho=\rho(\gamma_1,\gamma_2,C,\varepsilon)>0$ with the following property.

For any sequence $\{a_n\}_{n\in\NN}\subset \RR$ satisfying:

\begin{itemize}

\item $|a_n|\le C$,

\item there is a subset $\LL\subset \NN$ satisfying $\liminf_{n\to +\infty}\frac{1}{n}\#\{[0,n-1]\cap \LL\}>1-\rho$ such that $a_n\le \gamma_1$ for any $n\in\LL$,

\end{itemize}
then there is a subset $\JJ\subset\NN$ satisfying $\limsup_{n\to +\infty}\frac{1}{n}\#\{[0,n-1]\cap \JJ\}>1-\varepsilon$ such that for any $j\in\JJ$, one has that
$$\sum_{i=0}^{n-1}a_{i+j}\le n\gamma_2,~~~\forall n\in\NN.$$

\end{Lemma}

\begin{proof}[Proof of Proposition~\ref{Pro:uniform-pesin-block-random-perturbation}]
We apply Lemma~\ref{Lem:pliss-like} to put
$$\gamma_1=-(\theta+\alpha)/2,~\gamma_2=-\alpha,~C=\max_{[\underline\omega,x]\in\Omega^\ZZ\times M}|\log\|DG([\underline\omega,x])\||.$$
For any $\varepsilon>0$, take $\varepsilon'$ such that $(1-\varepsilon')^2>1-\varepsilon$ and fix $\rho=\rho(\gamma_1,\gamma_2,C,\varepsilon')>0$ as in Lemma~\ref{Lem:pliss-like}.

\begin{Claim}
There is $\ell\in\NN$ such that for any $G$-invariant measure $\eta_N$, which close to $\eta$, one also has that
$$\eta_N(\Lambda_{\ell,1}^G(E,(\theta+\alpha)/2))>1-\rho\varepsilon'.$$

\end{Claim}

\begin{proof}[Proof of the Claim]

By Lemma \ref{Lem:find-integer}, there exists $\ell \in \NN$ such that
$$
\eta(\Lambda^G_{\ell,1}(E,(2\theta+\alpha)/3))>1-\rho\varepsilon'.
$$
Since $\Lambda^G_{\ell,1}(E,(2\theta+\alpha)/3)\subset\left\{[\underline\omega,x]\in\Lambda^G:\|DG^{-\ell}|_{E([\underline\omega,x])}\|< {\rm e}^{-(\theta+\alpha)\ell/2}\right\}$, one has that
$$
\eta(\{[\underline\omega,x]\in\Lambda^G:\|DG^{-\ell}|_{E([\underline\omega,x])}\|< {\rm e}^{-(\theta+\alpha)\ell/2}\})>1-\rho\varepsilon'.
$$

Now for a sequence of $G$-invariant measures $\{\eta_n\}$ such that $\lim_{n\to\infty}\eta_n=\eta$, by the fact that $\{[\underline\omega,x]\in\Lambda^G:\|DG^{-\ell}|_{E([\underline\omega,x])}\|<{\rm e}^{-(\theta+\alpha)\ell/2 }\}$ is an open set, one has that
\begin{eqnarray*}
&&\liminf_{n\to\infty}\eta_n(\{[\underline\omega,x]\in\Lambda^G:\|DG^{-\ell}|_{E([\underline\omega,x])}\|< {\rm e}^{-(\theta+\alpha)\ell/2}\})\\
&\ge& \eta(\{[\underline\omega,x]\in\Lambda^G:\|DG^{-\ell}|_{E([\underline\omega,x])}\|< {\rm e}^{-(\theta+\alpha)\ell/2 }\})>1-\rho\varepsilon'.
\end{eqnarray*}
Since $\{[\underline\omega,x]\in\Lambda^G:\|DG^{-\ell}|_{E([\underline\omega,x])}\|<{\rm e}^{-(\theta+\alpha)\ell/2 }\}\subset \Lambda_{\ell,1}^G(E,(\theta+\alpha)/2)$, one can conclude.
\end{proof}

It follows from the Birkhoff ergodic theorem we know for $\eta_N$ almost every $[\underline{\omega},x]$ the limit
$$
\varphi([\underline{\omega},x]):=\lim_{n\to\infty}\frac{1}{n}\#\left\{i:~0\le i\le n-1,~G^{-i \ell}([\underline\omega,x])\in\Lambda^G_{\ell,1}(E,(\theta+\alpha)/2)\right\}
$$
exists and 
$$
\int \varphi d\eta_N=\eta_N(\Lambda^G_{\ell,1}(E,(\theta+\alpha)/2)).
$$
Therefore, $\int \varphi d\eta_N>1-\rho\varepsilon'$ by the above claim.
Let $B=\left\{[\underline{\omega},x]: \varphi([\underline{\omega},x])>1-\rho\right\}$.
\begin{eqnarray*}
1-\eta_N(B) &=& \eta_N(\{[\underline{\omega},x]:1-\varphi([\underline{\omega},x])\ge \rho\}) \\
&\le & \frac{\int (1-\varphi) d\eta_N}{\rho}\\
&<& \frac{\rho \varepsilon'}{\rho}=\varepsilon'.
\end{eqnarray*}
Thus $\eta_N(B)>1-\varepsilon'$.
For any point $[\underline{\omega},x]\in B$, set
$$a_i=\frac{1}{\ell}\log\|DG^{-\ell}|_{E(G^{-i\ell}([\underline\omega,x]))}\|, ~~~\forall i\ge 0.$$
and 
$$
\LL=\{i\in \NN \cup \{0\}: G^{-i\ell}([\underline{\omega},x])\in \Lambda_{\ell,1}^G(E, (\theta+\alpha)/2)\}
$$
Then we have
\begin{itemize}
\item $|a_n|\le C$ for every $n\in \NN \cup \{0\}$;

\item For every $i\in \LL$, $a_i<\gamma_1=-(\theta+\alpha)/2$ and

\item $\lim_{n\to +\infty}\frac{1}{n}\#\{[0,n-1]\cap \LL\}=\varphi([\underline{\omega},x])>1-\rho.$
\end{itemize}

%

Thus, by applying Lemma~\ref{Lem:pliss-like}, there is a subset $\JJ\subset\NN\cup\{0\}$ such that 
\begin{itemize}
\item
for any $j\in\JJ$, one has that for any $n\in\NN$,
$$\sum_{i=0}^{n-1} a_{j+i}\le -n\alpha.$$
\item $\limsup_{n\to\infty}\frac{1}{n}\#\{[0,n-1]\cap \JJ\}>1-\varepsilon'.$

\end{itemize}
In other words, for any $j\in\JJ$,
$$\prod_{i=0}^{n-1}\|DG^{-\ell}|_{E(G^{-(i+j)\ell}([\underline\omega,x]))}\|\le {\rm e}^{-n\ell\alpha},~~~\forall n\in\NN.$$
Consequently, by applying the Birkhoff ergodic theorem, for almost every $[\underline{\omega},x]\in B$ we have
$$\lim_{n\to\infty}\frac{1}{n}\#\left\{i:~0\le i\le n-1,~G^{-i\ell}([\underline\omega,x])\in\Lambda^G_\ell(E,\alpha)\right\}>1-\varepsilon'.$$
Therefore, there exists a subset 
$$
B_m=\left\{[\underline{\omega},x]\in B: \frac{1}{m}\#\{i\in \{0,\cdots,m-1\}: G^{-i\ell}([\underline{\omega},x])\in \Lambda^G_\ell(E,\alpha)\}>1-\varepsilon' \right\}
$$
such that $\eta_N(B_m)>1-\varepsilon'$.
Thus
\begin{eqnarray*}
\eta_N(\Lambda_{\ell}^G(E,\alpha))&=&\int \frac{1}{m}\sum_{i=0}^{m-1}\chi_{\Lambda_{\ell}^G(E,\alpha)}(G^{-il}([\underline{\omega},x]))d\eta_N\\
&\ge & \int_{B_m}\frac{1}{m}\sum_{i=0}^{m-1}\chi_{\Lambda_{\ell}^G(E,\alpha)}(G^{-il}([\underline{\omega},x]))d\eta_N\\
&=& \int_{B_m}\frac{1}{m}\#\{i\in \{0,\cdots,m-1\}: G^{-i\ell}([\underline{\omega},x])\in \Lambda^G_\ell(E,\alpha)\}d\eta_N\\
&>&(1-\varepsilon')\eta_N(B_m)>(1-\varepsilon')^2,
\end{eqnarray*}
where we use the  $G^{-\ell}$-invariance of $\eta_N$ in the first equality.
By the choice of $\varepsilon'$, one gets
$$
\eta_N(\Lambda_{\ell}^G(E,\alpha))>(1-\varepsilon')^2>1-\varepsilon.
$$
The proof is complete now.
\end{proof}

\subsection{Consequences for one diffeomorphism}

As some consequence of Proposition~\ref{Pro:uniform-pesin-block-random-perturbation}, one has the following results about the random perturbation and the ergodic limit for one diffeomorphism.
\begin{Proposition}\label{Pro:uniform-pesin-block}
Assume that an attracting set $\Lambda$ of a $C^2$ diffeomorphism $f$ admits a dominated splitting $TM|_\Lambda=E\oplus_\succ E^c\oplus_\succ F$ with $\dim E^c=1$. 
Assume that there is a regular random perturbation $\{(G,\nu_n)\}_{n\in\NN}$ generated by $\{(\Omega,\ell,\nu_n)\}_{n\in\NN}$ of $f$ such that
\begin{itemize}

\item Each random dynamical system $(G,\nu_n)$ has an ergodic stationary measure $\mu_n$ such that $\lim_{n\to\infty}\mu_n=\mu$.

\end{itemize}

If there is a constant $\alpha>0$ such that 
$$\inf\{\int\log\|Df|_{E^c}\|d\nu:~\nu~\textrm{is an ergodic component of }\mu\}>\alpha,$$
then for any $\varepsilon>0$, there is $\ell=\ell(\varepsilon)>0$ such that
$$\liminf_{n\to\infty}\mu^G_n(\Lambda^G_\ell(E\oplus E^c,\alpha))>1-\varepsilon.$$

\end{Proposition}

\begin{proof}
Suppose that $\ell(\omega_f)=f$. Note that $\mu$ can be lifted to be a measure on $\{\underline\omega_f\}\times M$ and we have that $\mu_n^G\to\mu^G$ as $n\to\infty$ by Corollary~\ref{Cor:lifted-limit}. Moreover, by Proposition~\ref{Pro:extend-dominated}, the support of $\mu_n^G$ admits the same kind of dominated splitting for $n$ large enough. After the lift, one has that any ergodic component of $\mu^G$ has its Lyapunov exponent larger than $\alpha$. Thus, one can apply Proposition~\ref{Pro:uniform-pesin-block-random-perturbation} to conclude.

\end{proof}

%
%

%
%
%
%
%

The following result is some corollary of Proposition~\ref{Pro:uniform-pesin-block-random-perturbation}:

\begin{Corollary}\label{Cor:one-diff-pesin-block}
Assume that $\Lambda$ is an attracting set of a $C^2$ diffeomorphism $f$ with a partially hyperbolic splitting $TM|_\Lambda=E^u\oplus_\succ E_1^c\oplus_\succ \cdots\oplus_\succ E_k^c\oplus_\succ E^s$ with $\dim E_i^c=1$, for $1\le i\le k$. Assume that $\{\mu_n\}\subset \cG_j$ is a sequence of ergodic measures and $\lim_{n\to\infty}\mu_n=\mu$. If there is $\alpha>0$ such that 
$$\inf\{\int\log\|Df|_{E^c_j}\|d\nu:~\nu~\textrm{is an ergodic component of }\mu\}>\alpha,$$
then for any $\varepsilon>0$, there is $\ell=\ell(\varepsilon)>0$ such that
$$\liminf_{n\to\infty}\mu_n(\Lambda_\ell(E_j^c,\alpha))>1-\varepsilon,$$
where $\Lambda_\ell(E_j^c,\alpha)=\{x\in\Lambda:~\prod_{i=0}^{n-1}\|Df^{-\ell}|_{E^c(f^{-i \ell}(x))}\|\le {\rm e}^{-\alpha \ell n},~\forall n\in\NN\}.$

\end{Corollary}

\begin{proof}

The dynamics of one diffeomorphism can be embedded into an extended dynamical system $G$ generated by $(\Omega,\ell)$ such that $\ell(\omega_f)=f$. One applies Corollary~\ref{Cor:lifted-limit} and Proposition~\ref{Pro:uniform-pesin-block-random-perturbation} to take $E=E^u\oplus_\succ E_1^c\oplus_\succ \cdots\oplus_\succ E_{i-1}^c$, $E^c=E_i^c$ and $F=E_{i+1}^c\oplus_\succ \cdots\oplus_\succ E_k^c\oplus_\succ E^s$ and identify $\Lambda_\ell(E_j^c,\alpha)$ and $\Lambda^G_\ell(E_j^c,\alpha)\cap\{\omega_f\}^\ZZ\times M$.
\end{proof}

\section{The disintegration along measurable partitions subordinate to unstable manifold}\label{Sec:disintegration}
Some definitions and results in Section~\ref{Sec:gibbs-cu} can be regarded as some special case of this section since the dynamics of one diffeomorphism can be embedded in the extended dynamical system $G$.
\subsection{Plaque families for the extended dynamical systems}

\begin{Definition}\label{Def:plaque-family-extended}
Assume that $\Lambda^G\subset\Omega^\ZZ\times M$ is a compact $G$-invariant set and $E\subset TM|_{\Lambda^G}$ is an invariant sub-bundle. A \emph{plaque family} of $E$, which is denoted by $\{W^E([\underline\omega,x])\}_{[\underline\omega,x]\in\Lambda^G}$, is a family of embedded sub-manifolds of dimension $\dim E$, each one is diffeomorphic to the unit ball in $\RR^{\dim E}$, and has the following properties:
\begin{itemize}

\item $W^E([\underline\omega,x])\subset \{\underline\omega\}\times M$ for any $[\underline\omega,x]\in\Omega^\ZZ\times M$;

\item For any point $[\underline\omega,x]\in\Lambda^G$, one has $T W^E([\underline\omega,x])|_{[\underline\omega,x]}=E([\underline\omega,x])$;

\item For any neighborhood $U\subset W^E(G([\underline\omega,x]))$ of $[\underline\omega,x]\in\Lambda^G$, there is a neighborhood $V$ of $[\underline\omega,x]$ in $W^E([\underline\omega,x])$ such that $G(V)\subset U$.

\end{itemize}
Denote by $W^E_\varepsilon([\underline\omega,x])$ the $\varepsilon$-neighborhood of $[\underline\omega,x]$ in $W^E([\underline\omega,x])$. The last property can be represented as: for any $\varepsilon>0$, there is $\delta>0$ such that for any $[\underline\omega,x]\in\Lambda^G$, one has $G(W^E_\delta([\underline\omega,x]))\subset W^E_\varepsilon(G([\underline\omega,x]))$.

\end{Definition}

In fact, one can require some higher regularity along plaque families. Generally, one can only increase a little bit of the  regularity in the dominated case. We will give a stronger notion called \emph{$(1+\alpha)$-domination}. A dominated splitting $E\oplus_\succ F$ on $\Lambda^G$ is said to be a \emph{$(1+\alpha)$-dominated splitting} if there are constants $C>0$ and $\lambda\in(0,1)$, one has for any $[\underline\omega,x]\in\Lambda^G$ and any $n\in\NN$,
$$\|DG^n|_{F([\underline\omega,x])}\|^{1+\alpha}.\|DG^{-n}|_{E(G^n([\underline\omega,x]))}\|\le C\lambda^n,~~~\|DG^n|_{F([\underline\omega,x])}\|.\|DG^{-n}|_{E(G^n([\underline\omega,x]))}\|^{1+\alpha}\le C\lambda^n.$$

Since the norms of the derivatives are uniformly bounded, one has the following lemma, whose proof could be an exercise.
\begin{Lemma}\label{Lem:alpha-dominate}
If $\Lambda^G$ is a compact $G$-invariant set with a dominated splitting $E\oplus_{\succ} F$, then there is $\alpha>0$ (possibly small) such that $E\oplus_\succ F$ is a $(1+\alpha)$-dominated splitting.

\end{Lemma}

For dominated splittings, one has the following plaque family theorem \cite[Theorem 5.5]{HPS77}: 
\begin{Theorem}\label{Thm:plaque-family-extended}
Assume that $\Lambda^G\subset\Omega^\ZZ\times M$ is a compact invariant set with a dominated splitting $TM|_{\Lambda^G}=E\oplus_\succ F$. Then there are plaque families tangent to $E$ and $F$. Moreover, given $\alpha\in(0,1)$, if the splitting is $(1+\alpha)$-dominated, then the plaques $W^E$ and $W^F$ can be chosen in the class of $C^{1+\alpha}$ sub-manifolds and varies continuously in the $C^{1+\alpha}$-topology with respect to the base points.

More precisely, for the bundle $E$, there is a continuous map $\Theta:~\Lambda^G\to {\rm Emb}^{r}(\DD^E,~\Omega^\ZZ\times M)$, where
\begin{itemize}
\item $r=1$ or $r=1+\alpha$ depending that we are under the assumption of domination or $(1+\alpha)$-domination, respectively.

\item $\DD^E$ is the unit disc contained in $\RR^E$,  ${\rm Emb}^r(\DD^E,~\Omega^\ZZ\times M)$ is the space of $C^r$ embeddings satisfying the image of each embedding is contained in some $\{\underline\omega\}\times M$.

\end{itemize}
such that for any $[\underline\omega,x]\in\Lambda^G$, one has that $W^E([\underline\omega,x])=\Theta([\underline\omega,x])(\DD^u)$.

One has a similar description for the plaque family of $F$.

\end{Theorem}

One has the existence of unstable manifolds in the dominated case. Its proof is almost the same as in the deterministic case. One can see \cite[Section 8]{ABC11} for instance.
\begin{Lemma}\label{Lem:dominated-unstable-extended}
Assume that $\Lambda^G\subset\Omega^\ZZ\times M$ is a compact $G$-invariant set with a dominated splitting $TM|_{\Lambda^G}=E\oplus_\succ F$.
Given $\ell\in\NN$ and $\lambda\in(0,1)$, there is $\delta=\delta(\ell,\lambda)>0$ such that for any point $[\underline\omega,x]\in\Lambda^G$, if 
$$\prod_{i=0}^{n-1}\|DG^{-\ell}|_{E(G^{-i\ell}([\underline\omega,x]))}\|\le\lambda^n,~~~\forall n\in\NN,$$
then $W^E_\delta([\underline\omega,x])$ is contained in the unstable manifold of $[\underline\omega,x]$; more precisely, there are constants $C=C(\ell,\lambda)>0$ and $\lambda_*=\lambda_*(\ell,\lambda)\in(0,1)$ such that for any $[\underline\omega,y],[\underline\omega,z]\in W^E_\delta([\underline\omega,x])$, one has that
$$d(G^{-n}([\underline\omega,y]),G^{-n}([\underline\omega,z]))\le C\lambda_*^n d([\underline\omega,y],[\underline\omega,z]).$$

\smallskip

Assume that $\mu$ is an ergodic measure of $G$ supported on $\Lambda^G$, and all Lyapunov exponents of $\mu$ along $E$ are positive. Then there is a positive $\mu$-measurable function $\delta([\underline\omega,x])$ for $\mu$-almost every point $[\underline\omega,x]$ such that $W^E_{\delta([\underline\omega,x])}([\underline\omega,x])$ is contained in the unstable manifold of $[\underline\omega,x]$.
\end{Lemma}

Note that as a consequence of Lemma~\ref{Lem:dominated-unstable-extended}, one has the following estimate on the size of unstable manifolds on a Pesin block. The proof is omitted.

\begin{Corollary}\label{Cor:uniform-size-extended}
Assume that $\Lambda^G$ is a compact $G$-invariant set with a dominated splitting $TM|_{\Lambda^G}=E\oplus_\succ F$. Given $\ell\in\NN$ and $\alpha>0$, there is $\delta=\delta(\ell,\alpha)>0$ such that $W^E_{\delta}([\underline\omega,x])$ is contained in the unstable manifold of $[\underline\omega,x]$ for any $[\underline\omega,x]\in\Lambda_{\ell}^G(E,\alpha)$.

\end{Corollary}

\subsection{The local foliated chart}
\begin{Notation}
Given $\delta\in(0,1]$, denote by
$\DD^E(\delta)=\{x\in\RR^{\dim E},~\|x\|\le\delta\}$ and $\DD^E=\DD^E(1)$.
\end{Notation}
We give some criteria to show the absolutely continuous property of the conditional measures. 

\begin{Definition}\label{Def:foliated-chart}
Assume that $\Lambda^G$ is a compact $G$-invariant set with a dominated splitting $TM|_{\Lambda^G}=E\oplus_\succ F$, and $\Gamma$ is a compact metric space. 

\smallskip

A \emph{foliated chart} associated to a set $\Gamma$ is a map $\Phi:~\Gamma\times \DD^{E}\mapsto \Omega^\ZZ\times M$ such that 
\begin{enumerate}
\item For any $p\in\Gamma$, $\Phi$ induces a map $\Phi_p:~\DD^E\to \Omega^\ZZ\times M$. $\Phi_p$ is a diffeomorphism.

\item $\Phi_p(\DD^E)$ is contained in a plaque tangent to $E$. 

\item $\Phi_p$ is continuous w.r.t. $p$ in the $C^1$ topology.

\item The imagine of $\Phi_p$ and the imagine of $\Phi_q$ are pairwise disjoint for $p\neq q$.

\end{enumerate}

A foliated chart induces a measurable partition, and Lebesgue measures on each element of the measurable partition. The image the map $\Phi$ is also denoted by $\Phi$. For any $p\in\Gamma$, the image of the map $\Phi_p$ is also denoted by $\Phi_p$. The projection from $\Phi$ to $\Gamma$ is denoted by $\pi$. Note that $\pi$ is continuous.
\end{Definition}
For any Borel measure $\mu$, denote the quotient measure on $\Gamma$ by $\widehat\mu=\pi_*(\mu)$. A family of conditional measures $\{\mu_{p}\}_{p\in\Gamma}$ is defined for $\widehat\mu$-almost every $p\in\Gamma$.
See \cite[Section C.6]{BDV05} and \cite[Section 1]{Rok67} for more details.

The following Lemma~\ref{Lem:measure-disintegration} gives a criterion for the conditional measures that are absolutely continuous w.r.t. Lebesgue measures. One can see \cite[Proposition 7.3]{Via99} for the proof of Lemma~\ref{Lem:measure-disintegration}.

\begin{Lemma}\label{Lem:measure-disintegration}

For a measurable partition induced by a foliated chart $\Phi$ associated to $\Gamma$ and a Borel measure $\mu$ on $\Phi$, if there is $C>0$ such that for any open set $A\subset \DD^E$, one has the following properties:
\begin{itemize}

\item $\mu(A\times\xi)\le C\widehat\mu(\xi){\rm Leb}(A)$, for any open set $\xi\subset\Gamma$ with $\widehat\mu(\partial\xi)=0$,

\end{itemize}
then the conditional measures of $\mu$ associated to this foliated chart are absolutely continuous w.r.t. the Lebesgue measures and the densities are bounded by $C$.

\end{Lemma}

\subsection{Gibbs $E$-states for the extended dynamical system}
With the unstable manifold for almost every points, one can define the Gibbs $E$-states for the extended dynamical system $G$. Using Lemma~\ref{Lem:dominated-unstable-extended}, one can define a measurable partition $\mu$-subordinate to $W^{E,u}$.
\begin{Definition}\label{Def:subordinate-extended}

Assume that $\Lambda^G\subset\Omega^\ZZ\times M$ is a compact $G$-invariant set with a dominated splitting $TM|_{\Lambda^G}=E\oplus_\succ F$. Assume that $\mu$ is a $G$-invariant measure satisfying the Lyapunov exponents along $E$ of $\mu$-almost every point are all positive. A measurable partition $\xi$ is said to be \emph{$\mu$-subordinate to $W^{E,u}$} if for $\mu$-almost every point $[\underline\omega,x]$, $\xi([\underline\omega,x])$ is an open set contained in $W^{E}_{\delta([\underline\omega,x])}([\underline\omega,x])$, where $\delta$ is the measurable function as in Lemma~\ref{Lem:dominated-unstable-extended}.

\bigskip

A $G$-invariant (not necessarily ergodic) measure $\mu$ supported on $\Lambda$ is a \emph{Gibbs $E$-state} if 
\begin{enumerate}
\item
For $\mu$-almost every point, its Lyapunov exponents along $E$ are all positive.

\item the conditional measures of $\mu$  are absolutely continuous  with respect to Lebesgue measures for any measurable partition $\mu$-subordinate to $W^{E,u}$.
\end{enumerate}

\bigskip

When $E$ is uniformly expanded\footnote{We say that $E$ is uniformly expanded on $\Lambda^G$, if there are constants $C>0$ and $\lambda\in(0,1)$ such that for any $[\underline\omega,x]\in\Lambda^G$ and any $n\in\NN$ such that $\|DG^{-n}|_{E([\underline\omega,x])}\|\le C\lambda^n$.} by $DG$, a Gibbs $E$-state is also called a \emph{Gibbs $u$-state} (as in the deterministic case).

\end{Definition}

One has the following result, whose proof is direct and omitted.

\begin{Lemma}\label{Lem:correspondence}
Let $G$ be the extended dynamical system generated by $(\Omega,\ell)$. Assume that there is $\omega_f\in\Omega$ such that $\ell(\omega_f)=f$. Assume that $\Lambda$ is a compact invariant set of $f$ with a dominated splitting $TM|_\Lambda=E\oplus F$ and $\mu$ is an invariant measure supported on $\Lambda$. Then $\mu$ is a Gibbs $E$-state if and only if $\mu^G$ is a Gibbs $E$-state for $G$.

\end{Lemma}

Recall that
$$\Lambda^G_\ell(E,\alpha)=\{[\underline\omega,x]\in\Lambda^G:\prod_{i=0}^{n-1}\|DG^{-\ell}|_{E(G^{-i \ell}([\underline\omega,x]))}\|\le {\rm e}^{-\alpha \ell n},~\forall n\in\NN\}.$$
The main result in this Section is:
\begin{Theorem}\label{Thm:criterion-gibbs-E-extended}
Assume that $\eta$  is a $G$-invariant measure and is supported on a compact invariant set $\Lambda^G\subset \Omega^\ZZ\times M$ with a dominated splitting $TM|_{\Lambda^G}=E\oplus_\succ F$. Assume that $\{\eta_n\}$ is a sequence of ergodic Gibbs $E$-states with the following properties:
\begin{itemize}
\item $\lim_{n\to\infty}\eta_n=\eta$.
\item There is a constant $\alpha>0$ such that for any $n\in\NN$, the Lyapunov exponents of $\eta_n$ along $E$ are larger than $\alpha>0$.
\item For any $\varepsilon>0$, there is $\ell\in\NN$ such that for any $n$ large enough, one has
$$\eta_n(\Lambda_\ell^G(E,\alpha))\ge 1-\varepsilon.$$
\end{itemize}
Then $\eta$ is a Gibbs $E$-state. 

\end{Theorem}

As a direct application of Theorem~\ref{Thm:criterion-gibbs-E-extended} in the uniform case, one has the following corollary:

\begin{Corollary}\label{Cor:gibbs-u-compact}
Assume that $\eta$  is a $G$-invariant measure and is supported on a compact invariant set $\Lambda^G\subset \Omega^\ZZ\times M$ with a dominated splitting $TM|_{\Lambda^G}=E^{uu}\oplus_\succ F$, where $E^{uu}$ is uniformly expanded by $DG$. If $\{\eta_n\}$ is a sequence of Gibbs $u$-states of $G$ and $\lim_{n\to\infty}\eta_n=\eta$, then $\eta$ is a Gibbs $u$-state.

\end{Corollary}
\begin{proof}
When $E^{uu}$ is uniformly expanded, then it is clear that there is $\alpha>0$ such that the Lyapunov exponents along $E^{uu}$ of any ergodic measure are larger than $\alpha$. Moreover, there is $\ell\in\NN$ such that $\Lambda^G=\Lambda_\ell^G(E^{uu},\alpha)$.
\end{proof}

Another consequence of Theorem~\ref{Thm:criterion-gibbs-E-extended} is the following deterministic version.

\begin{Corollary}\label{Cor:criterion-gibbs-E}
Assume that $f$ is a $C^2$ diffeomorphism, $\mu$  is an $f$-invariant measure and is supported on a compact invariant set $\Lambda\subset M$ with a dominated splitting $TM|_{\Lambda}=E\oplus_\succ F$. Assume that $\{\mu_n\}$ is a sequence of ergodic Gibbs $E$-states with the following properties:
\begin{itemize}
\item $\lim_{n\to\infty}\mu_n=\mu$.
\item There is a constant $\alpha>0$ such that for any $n\in\NN$, the Lyapunov exponents of $\mu_n$ along $E$ are larger than $\alpha>0$.
\item For any $\varepsilon>0$, there is $\ell\in\NN$ such that for any $n$ large enough, one has
$$\mu_n(\Lambda_\ell(E,\alpha))\ge 1-\varepsilon.$$
\end{itemize}
Then $\mu$ is a Gibbs $E$-state.

\end{Corollary}

\begin{proof}
Let $G$ be the extended dynamical system generated by $(\Omega,\ell)$. Assume that there is $\omega_f\in\Omega$ such that $\ell(\omega_f)=f$. Take $\eta_n=\mu_n^G$ for $n\in\NN$ and $\eta=\mu^G$. By Corollary~\ref{Cor:lifted-limit}, one has that $\lim_{n\to\infty}\eta_n=\eta$. By Lemma~\ref{Lem:extend-bundle}, one has that
\begin{itemize}
\item for any $n\in\NN$, $\eta_n$ and $\mu_n$ has same Lyapunov exponents. Hence the Lyapunov exponents of $\eta_n$ along $E$ are all larger than $\alpha$.

\end{itemize}
By Lemma~\ref{Lem:correspondence}, for any $n\in\NN$, $\eta_n$ is a Gibbs $E$-state for $G$ since $\mu_n$ is a Gibbs $E$-state for $f$. Since $\Lambda_\ell^G(E,\alpha)\supset\{\omega_f\}^\ZZ\times \Lambda_\ell(E,\alpha)$, one has that 
for any $\varepsilon>0$, there is $\ell\in\NN$ such that for any $n$ large enough, one has
$$\eta_n(\Lambda_\ell^G(E,\alpha))\ge 1-\varepsilon.$$

By Theorem~\ref{Thm:criterion-gibbs-E-extended}, $\eta=\mu^G$ is a Gibbs $E$-state. By applying Lemma~\ref{Lem:correspondence} again, one has that $\mu$ is a Gibbs $E$-state for $f$.
\end{proof}

We need the following result from Liu and Qian \cite[Chapter VI:~Proposition 2.2 and Corollary 8.1]{LiQ95}. We restate it as the following form.

\begin{Theorem}\label{Thm:liu-qian-density}
Assume that $\Lambda^G$ be a compact $G$-invariant set with a dominated splitting $TM|_{\Lambda^G}=E\oplus_\succ F$.
Let $\eta$ be a Gibbs $E$-state supported on $\Lambda^G$. Denote by
$$J^E([\underline\omega,x])=|{\rm Det}DG|_{E([\underline\omega,x])}|,~~~\forall [\underline\omega,x]\in\Lambda^G.$$

Then there exists the measurable partition $\xi$ that is $\eta$-{subordinate} to $W^{E,u}$. Moreover, for any such measurable partition $\xi$, for $\widehat\mu$-almost every $\xi([\underline\omega,x])$, one has 
$$
\frac{\rho([\underline{\omega},y])}{\rho([\underline{\omega},z])}=\prod_{j=1}^{+\infty}\frac{J^E(G^{-j}([\underline{\omega},z]))}{J^E(G^{-j}([\underline{\omega},y]))}, ~~\mu_{\xi([\underline\omega,x])}-\text{almost every}~ [\underline\omega,y],[\underline\omega,z] \in \xi([\underline{\omega},x]),
$$
where $\rho$ be the density of $\mu_\xi$ with respect to the Lebesgue measure on $\xi$.
\end{Theorem}

Based on Corollary~\ref{Cor:uniform-size-extended}, one can define the notion ``the disintegration of $\mu$ on $W^u_{loc}(\Lambda^G_\ell)$''. We first give some construction of the foliated chart. Recall that the plaque families are given by the map $\Theta$ as in Theorem~\ref{Thm:plaque-family-extended}.

\begin{Lemma}\label{Lem:contruct-foliated-box}\footnote{For a set $A\subset\RR^{\dim E}$ and $v\in\RR^{\dim E}$, define $A+v=\{a+v,a\in A\}$.}
Assume that $\Lambda^G$ is a compact $G$-invariant set with a dominated splitting $TM|_{\Lambda^G}=E\oplus_\succ F$. Given $\ell\in\NN$ and $\alpha>0$, there are $\delta=\delta(\ell,\alpha)>0$ and $\beta=\beta(\ell,\alpha)\in(0,\delta/4)$ such that for any $[\underline\omega,x]\in \Lambda^G_\ell(E,\alpha)$,  there is a continuous map $v:~\Lambda^G_\ell(E,\alpha)\to \DD^E(\delta/4)$ such that
$$\bigcup_{[\underline\omega',x']\in B([\underline\omega,x],\beta+\varepsilon)\cap \Lambda^G_\ell(E,\alpha)}\Theta([\underline\omega',x'])(\DD^E(\delta/2+\varepsilon)+v([\underline\omega',x']))$$
is the image of foliated chart $\Phi$ as in Definition~\ref{Def:foliated-chart} associated to compact set $\Gamma(\delta,\beta+\varepsilon,[\underline\omega,x])$ for any $\varepsilon$ small enough  with the following precise properties:
\begin{itemize}
\item $\Gamma(\delta,\beta+\varepsilon,[\underline\omega,x])$ is chosen as
$$\Gamma(\delta,\beta+\varepsilon,[\underline\omega,x])=(\Omega^\ZZ\times \PP_M(W^F_\delta([\underline\omega,x])))\bigcap \left(\bigcup_{[\underline\omega',x']\in B([\underline\omega,x],\beta+\varepsilon)\cap \Lambda^G_\ell(E,\alpha)}\{W^E_{\delta/2}([\underline\omega',x']\}\right).$$
\item For any $[\underline\omega',x']\in\Gamma$, there is $[\underline\omega^*,x^*]\in\Lambda^G_\ell(E,\alpha)$ such that the image of $\Phi_{[\underline\omega',x']}$ is contained in $W^{E}_{\delta}([\underline\omega^*,x^*])$.

\end{itemize}
\end{Lemma}
\begin{proof}
By Corollary~\ref{Cor:uniform-size-extended}, there is $\delta=\delta(\ell,\alpha)>0$ such that for any point $[\underline\omega,x]\in\Lambda_\ell^G$, $W^{E}_\delta([\underline\omega,x])$ is contained in the unstable manifold of $[\underline\omega,x]$.

Choose $\beta>0$ that is much smaller than $\delta$, one has that for any $[\underline\omega',x']$ in the $\beta$-neighborhood of $[\underline\omega,x]$, $W^E_\delta([\underline\omega',x'])$ intersects $\{\underline\omega'\}\times \PP_M(W^F_\delta([\underline\omega,x]))$ transversely. The intersection point is denoted by $[\underline\omega',y]$. Take $\Gamma(\delta,\beta,[\underline\omega,x])$ to be the union of this kind of points.

\smallskip

The plaque family theorem (Theorem~\ref{Thm:plaque-family-extended}) in fact gives the foliated chart $\Phi$. More precisely, for the map $\Theta:~\Lambda^G\to {\rm Emb}^{r}(\DD^E,~\Omega^\ZZ\times M)$ as given in Theorem~\ref{Thm:plaque-family-extended}, one has that $W^E_\delta([\underline\omega',x'])=\Theta([\underline\omega',x'])(\DD^E(\delta))$. For $\beta>0$ small enough, one has that $[\underline\omega',y]$ is close to the center of $W^E_\delta([\underline\omega',x'])$. Assume that $[\underline\omega',y]=\Theta([\underline\omega',x'])(v([\underline\omega',y])))$ for some $v([\underline\omega',y])\in\DD^E$ close to $0$.

\medskip

Now one takes $\Phi([\underline\omega',y])(\DD^E(\delta/2))=\Theta([\underline\omega',x'])(\DD^E(\delta/2)+v([\underline\omega',y]))$.

Note that one can modify a little bit the size of the plaques and the neighborhood such that after the modification, it is still a foliated chart. Thus we introduce the small auxiliary constant $\varepsilon>0$.
\end{proof}
%
%
\begin{Definition}\label{Def:foliated-box-extended-pesin-block}
Assume that $\Lambda^G$ is a compact $G$-invariant set with a dominated splitting $TM|_{\Lambda^G}=E\oplus_\succ F$. Let $\mu$ be a $G$-invariant measure supported on $\Lambda^G$. Given $\ell\in\NN$ and $\alpha>0$, we say that \emph{the disintegration of $\mu$ on $W^u_{loc}(\Lambda^G_\ell(E,\alpha))$ is absolutely continuous w.r.t. Leb}
if for $\mu$-almost every $[\underline\omega,x]\in\Lambda^G_\ell(E,\alpha)$, for any foliated box $\Phi$ associated to $\Gamma(\delta,\beta,[\underline\omega,x])$ as in Lemma~\ref{Lem:contruct-foliated-box}, the conditional measures of $\mu|_{\Phi}$ along the canonical partition are absolutely continuous with respect to the Lebesgue measures along the elements of the partition.

\end{Definition}

\begin{Lemma}\label{Lem:bound-density}

Assume that $\Lambda^G$ is a compact $G$-invariant set with a dominated splitting $TM|_{\Lambda^G}=E\oplus_\succ F$. Given $\ell\in\NN$ and $\alpha>0$, there is $L=L(\ell,\alpha)>0$ such that for any Gibbs $E$-state $\mu$, for $\mu$-almost every point $[\underline\omega,x]\in\Lambda^G_\ell(E,\alpha)$, for the foliated chart $\Phi$ constructed as in Lemma~\ref{Lem:contruct-foliated-box}, for the measurable partition $\xi$ induced by the foliated chart $\Phi$, one has that
$$\frac{\rho([\underline\omega,y])}{\rho([\underline\omega,z])}\le L,~~~\mu_{\xi([\underline\omega,x])}-\text{almost every}~ [\underline\omega,y],[\underline\omega,z] \in \xi([\underline{\omega},x]),$$
where $\rho$ is the density of $\mu_\xi$ with respect to the Lebesgue measure on $\xi$.
\end{Lemma}
\begin{proof}

This uses the bundles are H\"older and the density estimation before. By Lemma~\ref{Lem:alpha-dominate}, one knows there is $\alpha_H>0$ such that $TM|_{\Lambda^G}=E\oplus_\succ F$ is in fact a $(1+\alpha_H)$-dominated splitting. By Theorem~\ref{Thm:plaque-family-extended}, the tangent spaces of the plaques are uniformly H\"older with exponent $\alpha_H$. Consequently, there is a constant $C_H>0$ such that $\log J^E$ is $(C_H,\alpha_H)$-H\"older along any plaque.

By Theorem~\ref{Thm:liu-qian-density}, the density function $\rho$ of disintegration with respect to the measurable partition induced by the foliated chart $\Phi$ has the following property: for $\widehat\mu$-almost every $[\underline\omega,x]\in\Gamma$, for $\mu_{[\underline\omega,x]}$-almost every $[\underline\omega,y],~[\underline\omega,z]\in \Phi_{[\underline\omega,x]}$, we have
$$\frac{\rho([\underline\omega,y])}{\rho([\underline\omega,z])}=\prod_{j=0}^{+\infty}\frac{J^E(G^{-j}([\underline{\omega},z]))}{J^E(G^{-j}([\underline{\omega},y]))}$$

By Lemma~\ref{Lem:dominated-unstable-extended}, one has the constant $C>0$ and $\lambda_*$ depending on $\ell$ and $\alpha$ such that for any $[\underline\omega,x]\in \Lambda^G_\ell(E,\alpha)$, for any $[\underline\omega,y],~[\underline\omega,z]\in W^E_\delta([\underline\omega,x])$,
$$d(G^{-n}([\underline\omega,y]),G^{-n}([\underline\omega,z]))\le C\lambda_*^n d([\underline\omega,y],[\underline\omega,z]).$$

Since the plaques are uniformly H\"older by Theorem~\ref{Thm:plaque-family-extended},  we have that
\begin{eqnarray*}
\prod_{j=0}^{+\infty}\frac{J^E(G^{-j}([\underline{\omega},z]))}{J^E(G^{-j}([\underline{\omega},y]))}&\le& \exp\{C_H\sum_{n=0}^\infty d(G^{-j}([\underline{\omega},z]),G^{-j}([\underline{\omega},y]))^{\alpha_H}\}\\
&\le& \exp\{C_H \sum_{n=0}^\infty C^{\alpha_H}(\lambda_*^{\alpha_H})^n\},
\end{eqnarray*}
It suffices to take
$$L=\exp\{C_H \sum_{n=0}^\infty C^{\alpha_H}(\lambda_*^{\alpha_H})^n\}.$$ 
\end{proof}

To verify an invariant measure $\mu$ is a Gibbs $E$-states, it suffices to verify this fact for an increasing sequence of Pesin blocks. The following Lemma~\ref{Lem:criterion-disintegration-extended} is folklore. 
\begin{Lemma}\label{Lem:criterion-disintegration-extended}
Assume that $\Lambda^G$ is a compact $G$-invariant set with a dominated splitting $TM|_{\Lambda^G}=E\oplus_\succ F$. If a $G$-invariant measure $\mu$ has the following properties:
\begin{itemize}

\item $\lim_{\ell\to\infty}\mu(\Lambda^G_\ell(E,\alpha))=1$,
\item The disintegration of $\mu$ on $W^u_{loc}(\Lambda^G_\ell(E,\alpha))$ is absolutely continuous w.r.t. Leb,

\end{itemize}
Then $\mu$ is a Gibbs $E$-state.

\end{Lemma}

\begin{proof}[Proof of Theorem~\ref{Thm:criterion-gibbs-E-extended}]
The strategy is to apply Lemma~\ref{Lem:criterion-disintegration-extended} to conclude. Now we prove that $\lim_{\ell\to\infty}\eta(\Lambda_\ell^G(E,\alpha))=1$. By the assumption, for any $\varepsilon>0$, there is $\ell\in\NN$ such that for all $n$ large enough, one has that $\eta_n(\Lambda_\ell^G(E,\alpha))\ge 1-\varepsilon$. Since $\Lambda_\ell^G(E,\alpha)$ is a compact set, one has that $\eta(\Lambda_\ell^G(E,\alpha))\ge \limsup_{n\to\infty}\eta_n(\Lambda_\ell^G(E,\alpha))>1-\varepsilon$. By the arbitrariness of $\varepsilon$, one has that $\lim_{\ell\to\infty}\eta(\Lambda_\ell^G(E,\alpha))=1$.

\bigskip

\begin{Claim}
There are finitely many foliated charts $\{\Phi^i\}_{i=1}^n$ associated to $\{\Gamma(\delta,\beta,[\underline\omega^i,x^i])\}$ as in Lemma~\ref{Lem:contruct-foliated-box} having the following properties:
\begin{itemize}

\item For each $1\le i\le n$, one has that 
$$\eta(\Phi^i)>0,~~~\eta(\partial\Phi^i)=0.$$

\item $\eta(\Lambda^G_\ell(E,\alpha)\setminus\cup_{1\le i\le n}\Phi^i)=0$.
\end{itemize}

\end{Claim}
\begin{proof}[Proof of the Claim]
For any point $[\underline\omega,x]\in\Lambda^G_\ell(E,\alpha)$ contained in the support of $\mu$, one can construct a foliation chart $\Phi$ associated to $\Gamma(\delta,\beta+\varepsilon,[\underline\omega,x])$. One can modify $\varepsilon$ a little bit such that $\eta(\partial\Phi)=0$. Since $\Lambda^G_\ell(E,\alpha)$ is compact, one can find finitely many $\{\Phi^i\}$ whose interiors cover the intersection of $\Lambda^G_\ell(E,\alpha)$ and the support of $\eta$.
\end{proof}
%
%
%
%
%

Now for each $\Phi\in\{\Phi^1,\Phi^2,\cdots,\Phi_n\}$, since $\eta(\partial\Phi)=0$, one has that $\lim_{n\to\infty}\eta_n(\Phi^i)=\eta(\Phi^i)>0$. Moreover, since $\eta(\partial\Phi)=0$, one has that $\eta_n|_{\Phi}\to\eta|_{\Phi}$ in the weak-* topology.

\bigskip

\begin{Claim} 
For any open set $\gamma\subset\Gamma$ whose boundary has zero $\widehat\eta$-measure
, one has that $$\limsup_{n\to\infty}\widehat\eta_n(\gamma)\le\widehat\eta(\gamma).$$
\end{Claim}
\begin{proof}[Proof of the Claim]
Since the boundary of $\gamma$ has zero $\widehat\eta$-measure, one has that
$$\widehat\eta(\gamma)=\widehat\eta(\overline\gamma)=\eta(\Phi(\overline\gamma\times \DD^E)).$$
Since $\overline\gamma\times \DD^E$ is a compact set, one has that
$$\eta(\Phi(\overline\gamma\times \DD^E))\ge\limsup_{n\to\infty}\eta_n(\Phi(\overline\gamma\times \DD^E))=\limsup_{n\to\infty}\widehat\eta_n(\overline\gamma)\ge\limsup_{n\to\infty}\widehat\eta_n(\gamma).$$
Thus one can conclude.
\end{proof}

Choose an open set $\gamma\subset\Gamma$ satisfying $\widehat\eta(\partial\gamma)=0$. For any open set $A\subset \DD^E$, one has that 
$$\eta(\Phi(\gamma\times A))\le\liminf_{n\to\infty}\eta_n(\Phi(\gamma\times A)).$$

By Theorem~\ref{Lem:bound-density}, one has that there is a constant $L$ depending on $\ell$, $\alpha$, but independent of $n$ such that for any $n\in\NN$, one has that
$$\eta_n(\Phi(\gamma\times A))\le L.\widehat\eta_n(\gamma). {\rm Leb}(A)$$
Consequently, by the above claim, one has that
$$\eta(\Phi(\gamma\times A))\le L.\widehat\eta(\gamma). {\rm Leb}(A)$$

By Lemma~\ref{Lem:measure-disintegration}, the disintegration of $\mu$ for this foliated chart is absolutely continuous with respect to the Lebesgue measure.

Since $\lim_{\ell\to\infty}\eta(\Lambda_\ell^G(E,\alpha))=1$, by Lemma~\ref{Lem:criterion-disintegration-extended}, one has that $\eta$ is a Gibbs $E$-state.
\end{proof}

\subsection{The applications of Theorem~\ref{Thm:criterion-gibbs-E-extended}: the proofs of Theorem~\ref{Thm:criterion-in-i} and Theorem~\ref{Thm:random-limit-G}}\label{Sec:prove-the-left}

We give the proof of Theorem~\ref{Thm:random-limit-G}.
\begin{proof}[Proof of Theorem~\ref{Thm:random-limit-G}]
Note that $\mu$ is a randomly ergodic limit. Assume that $\mu=\lim_{n\to\infty}\mu_n$, where $\mu_n$ is an ergodic stationary measure of a random dynamical system $(G,\nu_n)$, where $\{(G,\nu_n)\}_{n\in\NN}$ is a regular random perturbation of $f$.

By Lemma~\ref{Lem:invariant-measure-G}, for the extended dynamical system $G$, one has that $\lim_{n\to\infty}\mu_n^G=\mu^G$. 

Note also that $\mu_n^G$ are Gibbs $u$-states as in the proof of \cite[Proposition 5]{CoY05}.

\begin{Claim}
If any ergodic component of $\mu$ is not an SRB, then there is $i$ such that any ergodic component $\nu$ of $\mu$, one has that $\lambda_{i+1}^c(\nu)\ge 0$ but there is an ergodic component $\nu_-$ of $\mu$ such that $\lambda_{i+2}^c(\nu_-)<0$. 
\end{Claim}
\begin{proof}[Proof of the Claim]
We have that $\mu$ is a Gibbs $u$-state by Corollary~\ref{Cor:gibbs-u-compact}. See also \cite[Proposition 5]{CoY05}. Thus, any ergodic component $\nu$ of $\mu$ is also a Gibbs $u$-state by Proposition~\ref{Pro:property-u-state}. If $\lambda_1^c(\nu)<0$ for some ergodic component $\nu$ of $\mu$, then $\nu$ is an SRB measure by Lemma~\ref{Lem:component-srb}. This gives a contradiction. Thus $\lambda_1^c(\nu)\ge 0$ for any ergodic component $\nu$ of $\mu$. The maximal element of
$$\{j:~\lambda_{j+1}^c(\nu)\ge 0~\textrm{for any ergodic component}~\nu~\textrm{of }\mu\}$$
satisfies the property as in the Claim.
\end{proof}

Now one can assume that any ergodic component of $\mu$ is not an SRB measure.

\smallskip

By the above claim, one has that $\lambda_{i+1}^c(\nu)\ge 0$ for any ergodic component $\nu$ of $\mu$. Thus there is $\alpha>0$ (associated to the constant of the dominated splitting) such that $\lambda_i^c(\nu)>\alpha>0$ for any ergodic component $\nu$ of $\mu$. Thus, the same holds for $\mu^G$.

\smallskip

Take $E=E^u\oplus_\succ E_1^c\oplus_\succ\cdots\oplus_\succ E_i^c$. By Proposition~\ref{Pro:uniform-pesin-block}, one has for any $\varepsilon>0$ there is $\ell=\ell(\varepsilon)\in\NN$ such that $\liminf_{n\to\infty}\mu_n^G(\Lambda_\ell^G(E,\alpha))>1-\varepsilon$. Now one can apply Theorem~\ref{Thm:criterion-gibbs-E-extended} to conclude that $\mu^G$ is Gibbs $E$-state, hence so is $\mu$ by Lemma~\ref{Lem:correspondence}. Thus we have proved that $\mu\in\cG_i$.

\smallskip

There are several cases:

\begin{enumerate}

\item\label{i.zero} $\lambda_{i+1}^c(\nu_0)=0$ for some ergodic component $\nu_0$ of $\mu$.

\item\label{i.uniform} There is $\alpha>0$ such that $\lambda_{i+1}^c(\nu)>\alpha>0$ for any ergodic component $\nu$ of $\mu$.

\item\label{i.limit} $\lambda^{c}_{i+1}(\nu)>0$ for any ergodic component $\nu$ of $\mu$, but there is a sequence of ergodic components $\{\nu_n\}$ of $\mu$ such that $\lim_{n\to\infty}\lambda^c_{i+1}(\nu_n)=0$.

\end{enumerate}

In Case~\ref{i.zero}, by Lemma~\ref{Lem:component-srb}, one knows that $\nu_0$ is an SRB measure. This contradicts to the fact that we have assumed that no ergodic component of $\nu$ of $\mu$ is an SRB measure. Thus Case~\ref{i.zero} is impossible.

\smallskip

In Case~\ref{i.uniform}, by following the arguments above, one knows that $\mu\in\cG_{i+1}$. For completeness, we repeat the proof. Take $E'=E^u\oplus_\succ E_1^c\oplus_\succ\cdots\oplus_\succ E_i^c\oplus_{\succ}E_{i+1}^c$. By Proposition~\ref{Pro:uniform-pesin-block}, one has for any $\varepsilon>0$ there is $\ell=\ell(\varepsilon)\in\NN$ such that $\liminf_{n\to\infty}\mu_n^G(\Lambda_\ell^G(E',\alpha))>1-\varepsilon$. Now one can apply Theorem~\ref{Thm:criterion-gibbs-E-extended} to conclude that $\mu^G$ is Gibbs $E'$-state, hence so is $\mu$ by Lemma~\ref{Lem:correspondence}. Thus we have proved that $\mu\in\cG_{i+1}$. But there is an ergodic component $\nu_-$ of $\mu$ such that $\lambda_{i+2}^c(\nu_-)<0$, one has that $\nu_-$ is an SRB measure by Lemma~\ref{Lem:component-srb}. Thus Case~\ref{i.uniform} is impossible.

\smallskip

 In Case~\ref{i.limit}, one knows that $\mu\in\cG_i^0$ by definition. Thus one can conclude the theorem.
\end{proof}

We give the proof of Theorem~\ref{Thm:criterion-in-i}.

\begin{proof}[Proof of Theorem~\ref{Thm:criterion-in-i}]
Take $E=E^u\oplus_\succ E_i^c\oplus_\succ \cdots\oplus_\succ E_i^c$. By Corollary~\ref{Cor:one-diff-pesin-block}, for any $\varepsilon>0$, there is $\ell=\ell(\varepsilon)\in\NN$ such that
$$\liminf_{n\to\infty}\mu_n(\Lambda_\ell(E_j^c,\alpha))>1-\varepsilon.$$
Then one can apply Corollary~\ref{Cor:criterion-gibbs-E} directly to conclude.
\end{proof}

\appendix

\section{The absolute continuity of invariant manifolds}\label{App:invariant-manifold}

Let $W$ be an embedded manifold of $M$. A foliation $\cF$ of $W$ is \emph{absolutely coninuous} if for any two cross section $\Sigma_1$ and $\Sigma_2$ in $W$ that are close and transverse to the foliation $\cF$ in $W$, the holonomy map $h:\Sigma_1\to\Sigma_2$ defined by the foliation $\cF$ has the following property: $h_*({\rm Leb}_{\Sigma_1})$ is absolutely continuous with respect to ${\rm Leb}_{\Sigma_2}$.

A fundamental property of an absolutely continuous foliation is the following (one can see \cite[Lemma 3.4]{AVW15} for the proof):
\begin{Lemma}\label{Lem:absolutely-cont}
Assume that $W$ is an embedded sub-manifold of $M$ and $\cF$ is an absolutely continuous foliation of $W$. Then the conditional measures of the Lebesgue measure of $W$ with respect to the measurable partition associated to $\cF$ are absolutely continuous with respect to the Lebesgue measures of the leaves of $\cF$.

\end{Lemma}

About the plaque families, one has the following result (Lemma~\ref{Lem:absC}) on the absolute continuity. Recall that 

\begin{Lemma}\label{Lem:absC}
Assume that $f$ is a $C^2$ diffeomorphism and assume that $\Lambda$ is a compact $f$-invariant set with a dominated splitting $TM|_\Lambda=\Delta_1\oplus_\succ \Delta_2\oplus_\succ \Delta_3$. Given $\ell\in\NN$ and $\alpha>0$, there is $\delta=\delta(\ell,\alpha)$ such that for any point $x\in\Lambda_\ell(\Delta_2,\alpha)$, i.e., 
$$\prod_{i=0}^{n-1}\|Df^{-\ell}|_{\Delta_2(f^{-i \ell}(x))}\|\le {\rm e}^{-\alpha \ell n},~\forall n\in\NN$$
the foliation
$$\{W^{\Delta_1}_\delta(y):~y\in W^{\Delta_1\oplus\Delta_2}(x)\}$$
is an absolutely continuous foliation of $W^{\Delta_1\oplus\Delta_2}(x)$.
\end{Lemma}
\begin{proof}
We give a sketch of the proof. By relaxing the constants, for any point $x\in\Lambda_\ell(\Delta_2,\alpha)$, one has that
$$\prod_{i=0}^{n-1}\|Df^{-\ell}|_{\Delta_1\oplus_\succ\Delta_2(f^{-i \ell}(x))}\|\le {\rm e}^{-\alpha \ell n},~\forall n\in\NN$$
Thus, by Lemma~\ref{Lem:dominated-unstable}, there is $\delta=\delta(\ell,\alpha)$ such that $W^{\Delta_1\oplus\Delta_2}_\delta(x)$ is contained in the (exponentially) unstable manifold of $x$.

Thus, by reducing $\delta$ if necessary, for any point $y\in W^{\Delta_1\oplus\Delta_2}_\delta(x)$, $W^{\Delta_1}_\delta(y)$ is the stronger unstable manifold in $W^{\Delta_1\oplus\Delta_2}_\delta(x)$. The absolutely continuity follows from a similar argument in \cite[Chapter 11]{BaP02}.
\end{proof}

Now we can give the proof of Proposition~\ref{Proposition:several-G}.
\begin{proof}[Proof of Proposition~\ref{Proposition:several-G}]
It suffices to prove that for any $0\le i\le k-1$, one has that $\cG_i\supset\cG_{i+1}$. We set $E=E^{uu}\oplus_\succ E_1^c\oplus_\succ\cdots\oplus_\succ E_{i+1}^c$ and $\Delta=E^{uu}\oplus_\succ E_1^c\oplus_\succ\cdots\oplus_\succ E_{i}^c$. For any measure $\mu\in\cG_{i+1}$, one knows that
\begin{enumerate}
\item\label{i.positive-lya} $\mu$-almost every point has its Lyapunov exponents along $E$ are positive.

\item\label{i.disintegrationdd}The conditional measures of $\mu$ along $W^{E,u}$ are absolutely continuous w.r.t. Lebesgue.

\end{enumerate}
By Item~\ref{i.positive-lya}, there are $\ell\in\NN$ and $\alpha>0$ such that $x\in \Lambda_\ell(\alpha,E^c_{i+1})$. By Lemma~\ref{Lem:absC}, the foliation
$$\{W^{\Delta}_\delta(y):~y\in W^{E}(x)\}$$
is an absolutely continuous foliation of $W^{E}(x)$. Thus from Lemma~\ref{Lem:absolutely-cont}, the conditional measures of the Lebesgue measure on $W^E(x)$ along the foliation $\{W^{\Delta}_\delta(y):~y\in W^{E}(x)\}$ are Lebesgue measures. By Item~\ref{i.disintegrationdd} and the transitivity of conditional measures, one can conclude.
\end{proof}

\section{The proof of the Pliss-like lemma}\label{Sec:pliss-like}

\begin{proof}[Proof of Lemma~\ref{Lem:pliss-like}]
For any given $\varepsilon>0$, take
$$
0<\rho< \min\left\{1,\frac{(\gamma_2-\gamma_1)}{2(2C-\gamma_1)}, \frac{\gamma_2-\gamma_1}{C-\gamma_1}\varepsilon\right\}.
$$
The subset $\JJ\subset\NN$ is defined by
$$\JJ=\{j\in\NN:~\sum_{i=0}^{n-1}a_{i+j}\le n\gamma_2,~~~\forall n\in\NN\}.$$ 
We are going to prove that $\limsup_{n\to\infty}\frac{1}{n}\#(\JJ\cap[1,n])\ge 1-\varepsilon$. Fix $\gamma=(\gamma_1+\gamma_2)/2$.

\begin{Claim}
For any $L$ large enough, one has that
$$\sum_{i=1}^L a_i\le L\gamma.$$

\end{Claim}

\begin{proof}
Choose a large integer $L\in\LL$ such that $\rho L>1$. By the property of $\LL$, there are integers $\GG=\{n_1,n_2,\cdots,n_k\}\subset [1,L]$ such that $\#\GG\ge (1-\rho)L$ and $a_{n_i}<\gamma_1$. Thus, one has that
$$\sum_{i=1}^L a_i=\sum_{m\in \GG}a_m+\sum_{m\in[1,L]\setminus \GG}a_m\le \gamma_1(1-\rho)L+C(\rho L+1) \le ((2C-\gamma_1)\rho+\gamma_1)L\le \gamma L$$
when $\rho<(\gamma_2-\gamma_1)/2(2C-\gamma_1)$.
\end{proof}
From the above Claim, by the usual Pliss Lemma as in \cite{pliss}, one knows that $\JJ$ is a non-empty set with infinite cardinality.
%

To conclude, it suffices to prove that for some large $J\in\JJ$, one has that $\JJ\cap[1,J]\ge (1-\varepsilon)J$. We will prove by contradiction and assume that $\JJ\cap[1,J]< (1-\varepsilon)J$ for any large $J$. $[1,J]\setminus\JJ$ can be split into finitely many intervals $\{I_\alpha=[c_\alpha,d_\alpha)\}_{\alpha\in\cA}$ such that 
\begin{itemize}
\item $\sum_{m\in [c_\alpha,d_\alpha)}a_m\ge (d_\alpha-c_\alpha)\gamma_2$ for any $\alpha\in\cA$.
\item $\sum_{\alpha\in\cA}(d_\alpha-c_\alpha)\ge \varepsilon J$.

\end{itemize}
Set $\BB=\cup_{\alpha\in\cA}I_\alpha$. Since $\liminf_{n\to +\infty}\frac{1}{n}\#\{[0,n-1]\cap \LL\}>1-\rho$, for $J$ large enough, one has that $\# (\LL\cap [1,J])\ge (1-\rho)J$.

\begin{Claim} One has the following estimate:
$$\#(\BB\setminus\LL)\ge \frac{\gamma_2-\gamma_1}{C-\gamma_1}\#(\BB)\ge  \frac{\gamma_2-\gamma_1}{C-\gamma_1}\varepsilon J.$$
\end{Claim}
\begin{proof}
We have the following two estimates:
\begin{itemize}
\item $\sum_{i\in\BB}a_i>(\#\BB)\gamma_2$.

\item $\sum_{i\in\BB}a_i\le \sum_{i\in\BB\cap \LL}a_i+\sum_{i\in \BB\setminus\LL}a_i\le (\#(\BB\cap\LL))\gamma_1+(\#(\BB\setminus\LL))C=(\#\BB)\gamma_1+(\#(\BB\setminus\LL))(C-\gamma_1).$

\end{itemize}
By combining the above two inequalities one obtains that 
$\#(\BB\setminus\LL)\ge \frac{\gamma_2-\gamma_1}{C-\gamma_1}\#(\BB)$.
The last inequality follows from $\#\BB\ge \varepsilon J$.
\end{proof}

Consequently, we have that
\begin{eqnarray*}
\rho J&\ge&\#([1,J]\setminus\LL)\ge \#(\BB\setminus\LL)\\
&\ge&\frac{\gamma_2-\gamma_1}{C-\gamma_1}\#(\BB)\ge  \frac{\gamma_2-\gamma_1}{C-\gamma_1}\varepsilon J.
\end{eqnarray*}
This gives a contradiction since $\rho< {(\gamma_2-\gamma_1)}\varepsilon/{(C-\gamma_1)}$.
\end{proof}

\vskip 5pt

\noindent Yongluo Cao

\noindent School of Mathematical Sciences, Shanghai key Labaratory of PMMP

\noindent East China Normal University, Shanghai, 200062, P.R. China

\noindent School of Mathematical Sciences, Center for Dynamical Systems and Differential Equations

\noindent Soochow University, Suzhou, 215006, P.R. China

\noindent ylcao@suda.edu.cn

\vskip 5pt

\noindent Zeya Mi

\noindent College of Mathematics and Statistics

\noindent Nanjing University of Information Science and Technology, Nanjing 210044, China

\noindent mizeya@163.com

\vskip 5pt

\noindent Dawei Yang

\noindent School of Mathematical Sciences, Center for Dynamical Systems and Differential Equations

\noindent Soochow University, Suzhou, 215006, P.R. China

\noindent yangdw1981@gmail.com, yangdw@suda.edu.cn

\end{document}